%% file: main.tex
\documentclass[11pt,a4paper]{article}
\usepackage[T1]{fontenc}
\usepackage[utf8]{inputenc}
\usepackage{lmodern,amsmath,amssymb,amsthm,mathtools,amscd,mathrsfs}
\usepackage[margin=1in,headheight=14pt]{geometry}
\usepackage{microtype,booktabs,array,longtable,enumitem,fancyhdr,needspace}
\usepackage[hidelinks,unicode,pdfencoding=auto]{hyperref}
\usepackage{bookmark}
\hypersetup{pdftitle={Generalized Divisors and Canonical Coarsening on Tame Deligne-Mumford Curves},pdfauthor={Minghua Dou},pdfsubject={Algebraic geometry: generalized divisors, quotient curves, canonical lattices and conductor corrections},pdfkeywords={Deligne-Mumford stacks, generalized divisors, canonical modules, conductor, equivariant deformation}}
\hypersetup{bookmarksdepth=3}
\numberwithin{equation}{section}
\newtheorem{theorem}{Theorem}[section]
\newtheorem{proposition}[theorem]{Proposition}
\newtheorem{lemma}[theorem]{Lemma}
\newtheorem{corollary}[theorem]{Corollary}
\newtheorem{maintheorem}{Theorem}

\theoremstyle{definition}

\theoremstyle{remark}
\newtheorem{remark}[theorem]{Remark}
\newtheorem{question}[theorem]{Question}
\newcommand{\C}{\mathbb C}\newcommand{\kk}{\mathbb C}\newcommand{\kfield}{\mathbb C}
\newcommand{\A}{\mathbb A}\newcommand{\PP}{\mathbb P}
\newcommand{\OO}{\mathcal O}\newcommand{\XX}{\mathcal X}\newcommand{\M}{\mathcal M}
\newcommand{\cK}{\mathcal K}\newcommand{\cT}{\mathcal T}\newcommand{\cD}{\mathcal D}

\newcommand{\frs}{\mathfrak s}\newcommand{\frc}{\mathfrak c}\newcommand{\m}{\mathfrak m}
\newcommand{\la}{\lambda}\newcommand{\dd}{\mathrm d}\newcommand{\Om}{\Omega}
\newcommand{\Hom}{\operatorname{Hom}}\newcommand{\End}{\operatorname{End}}\newcommand{\Ext}{\operatorname{Ext}}
\newcommand{\Spec}{\operatorname{Spec}}\newcommand{\Spf}{\operatorname{Spf}}\newcommand{\Supp}{\operatorname{Supp}}
\newcommand{\Ann}{\operatorname{Ann}}
\newcommand{\coker}{\operatorname{coker}}
\newcommand{\length}{\operatorname{length}}\newcommand{\rk}{\operatorname{rk}}\newcommand{\wt}{\operatorname{wt}}
\newcommand{\tr}{\operatorname{tr}}
\newcommand{\Gr}{\operatorname{Gr}}\newcommand{\Quot}{\operatorname{Quot}}\newcommand{\Div}{\operatorname{Div}}
\newcommand{\Aut}{\operatorname{Aut}}
\newcommand{\RHom}{\mathbf R\!\operatorname{Hom}}
\newcommand{\anc}[1]{\nolinkurl{#1}}
\title{\textbf{Generalized Divisors and Canonical Coarsening\\on Tame Deligne--Mumford Curves}\\[0.45em]
\large Lattice fibres, invisible deformations, and conductor corrections}
\author{Minghua Dou}
\date{6 November 2026}
\begin{document}
\maketitle
\begingroup
\renewcommand{\thefootnote}{}
\makeatletter
\long\def\@makefntext#1{\noindent#1}
\makeatother
\begin{NoHyper}
\footnotetext{\textit{2020 Mathematics Subject Classification.} Primary 14A20, 14C20, 14H20; Secondary 14C05, 14B07, 13C14.\par
\textit{Keywords.} Deligne--Mumford stack, generalized divisor, canonical module, coarse moduli space, conductor, equivariant deformation.}
\end{NoHyper}
\endgroup
\begin{abstract}
We study effective generalized divisors on reduced tame Deligne--Mumford curves over $\C$ satisfying $G_1$ and $S_2$. Canonical coarsening maps divisor ideals to coarse canonical lattices; we describe its fibres by exact intervals of equivariant lattices. For flat projective finite-quotient families of geometrically reduced Gorenstein curves over locally Noetherian bases, with finite constant groups and the stated relative-duality and base-change hypotheses, we construct projective invariant-Quot schemes for fixed canonical lattices with finite flat quotients and prescribed quotient representations. A nontrivial marked deformation preserves the coarse ring, normalization, conductor and entire fixed canonical fibre, while a length-four isotypic jet detects its parameter. We compute the full normalization and conductor correction after support restriction on an explicit rank-three chart. As a module, the correction splits into one surface and two line canonical modules, with a common grading shift; the corresponding square-zero algebra extension is nonsplit. The full source--support pair is compared through second order. Exact computational certificates accompany the proofs.
\end{abstract}
\begingroup
\small
\makeatletter
\renewcommand{\l@section}{\@dottedtocline{1}{0em}{2.1em}}
\makeatother
\tableofcontents
\endgroup
\clearpage
\input{sections/01_introduction}
\input{sections/02_canonical_coarsening}
\input{sections/03_ordinary_coarsening}
\input{sections/04_length_one}
\input{sections/05_deformations}
\input{sections/06_conductor}
\input{sections/07_second_order}
\input{sections/08_perspectives}
\input{sections/acknowledgements}
\appendix
\input{sections/appendix_generators}
\input{sections/appendix_incidence}
\input{sections/appendix_verification}

\input{references}
\end{document}

%% file: sections/01_introduction.tex
\section{Introduction}\label{sec:introduction}

Generalized divisors extend the study of effective divisors to subschemes which need not be locally principal. On a reduced Gorenstein curve, an effective generalized divisor is an actual finite-colength ideal: its inclusion carries information beyond the isomorphism class of the underlying sheaf. On a Deligne--Mumford curve, stabilizers also act on the ideal and the canonical sheaf. Passage to the coarse curve retains only the corresponding invariant part.

We study this passage for tame quotient curves. The abbreviation \emph{DMH} refers to Deligne--Mumford stacks satisfying Hartshorne's conditions $G_1$ and $S_2$ on an \'etale atlas. For reduced curves, these conditions make the local rings Gorenstein. Canonical duality then provides a natural language for describing divisors, their contractions to the coarse curve, and the fibres of that contraction. The relation to the earlier preprint \cite{dou}, including the corrections to its foundational discussion, is explained in Section~\ref{sec:earlier-preprint}.

Let $\pi:\mathcal X\to C$ be the coarse morphism of a tame DMH curve and let $I\hookrightarrow\OO_{\mathcal X}$ be an effective generalized divisor. Two natural operations must be distinguished:
\begin{equation}\label{eq:intro-two-operations}
 I\longmapsto J=\pi_*I\subseteq\OO_C,
 \qquad
 I\longmapsto\Gamma_\omega(I)=\pi_*(I\otimes\omega_{\mathcal X})\subseteq\omega_C.
\end{equation}
The second operation is called \emph{canonical coarsening}. Its target is an actual canonical lattice. Even when $\mathcal X$ is Gorenstein, $C$ need not be Gorenstein; ordinary reflexivity and canonical reflexivity are then different. Replacing $J$ by its ordinary bidual can change both the ideal and the information recorded by a family.

We ask which effective generalized divisors lift a specified coarse canonical lattice, and which equivariant and deformation data this contraction retains. Three results address the question. A lattice-interval theorem gives a finite description of each fixed fibre. A family of quotient curves shows that an entire fixed canonical fibre can remain constant under a nontrivial deformation. Finally, an explicit calculation at a rank-three moduli point identifies the mixed-dimensional correction needed when normalization is followed by support restriction.

\subsection{Lattice fibres and fixed representations}
Work first with $\mathcal X=[\Spec R/G]$, where $R$ is a reduced Gorenstein curve ring over $\C$ and $G$ is finite. Write $B=R^G$, $W=\omega_R$ and $\Omega=\omega_B$. Finite duality identifies $W=\Hom_B(R,\Omega)$, with its natural $G$-action. A canonical lattice is an actual inclusion $K\subseteq\Omega$ of finite colength.

\begin{maintheorem}[Canonical lifting]\label{main:lifting}
The effective generalized-divisor ideals satisfying $(I\otimes W)^G=K$ correspond exactly to the equivariant lattices
\[
 RK\ \subseteq\ L\ \subseteq\ \Hom_B(R,K),\qquad I=L\otimes_R W^{-1}.
\]
The quotient $\Delta_K=\Hom_B(R,K)/RK$ has finite length and no invariants. For fixed quotient representation, the fibre is a projective scheme of stable submodules of $\Delta_K$.

For a projective finite-quotient family with the relative duality and base-change hypotheses specified in Section~\ref{sec:coarsening}, the same lifting problem has a projective invariant-Quot presentation using $\mathcal T_K=\coker(q^*K\to\omega_{Y/S})$. This presentation does not require $\mathcal T_K$ to be flat, or coinduction $\Hom_B(R,K)$ to commute with base change.
\end{maintheorem}

The first assertion is Theorem~\ref{lat-thm:interval}; the relative statement is Theorem~\ref{lat-thm:relative-quot}. Both are based on an explicit Reynolds-averaging adjunction. In particular, fixed representation is part of the moduli problem: a lattice may admit some lift but no lift of a specified representation.

For $m\ge3$, put
\[
R_m=\C[[x,y]]/(y^{2m-2}-x^{2m-1}),\qquad
G_m=\mu_{2m-3},\qquad\wt(x,y)=(1,2).
\]
We compute the length-one canonical-quotient problem with ordinary representation $\C[G_m]\oplus\mathbf1$. If $r=m-2$, the length-one canonical quotients annihilated by the coarse maximal ideal form $\PP^r$, and the entire specified lifting scheme is
\[
V(Bq_0)\subseteq\PP^1_{[A:B]}\times\PP^1_{[q_0:q_1]}.
\]
The map to $\PP^r$ is the rational normal curve $[A:B]\mapsto[A^r:A^{r-1}B:\cdots:B^r]$ after projection to the first factor. One component is contracted, its open affine line is the Cartier locus, and the other component consists of non-Cartier divisors. The proof works over nonreduced test bases. A moving-lattice example also shows explicitly why the coinduced endpoint cannot be used as an unconditional vector-bundle model in families.

\subsection{Ordinary coarsening and reflexivity}
The same cyclic plane branches have coarse rings $B_m=\C+s^m\C[[s]],\quad s=t^{2m-3}$, of Cohen--Macaulay type $m-1$. Thus the coarse singularity can be arbitrarily far from Gorenstein even though its stack cover is a plane-branch hypersurface. The existence of symmetric semigroup covers is known in much greater generality \cite{swanson}; the point here is to follow a specified divisor and its effective realizations through the quotient.

For the Cartier ideal $I_m=xR_m$ we compute
\[
J_m=I_m^{G_m}=(s^m,\ldots,s^{2m-2})B_m,
\qquad
J_m^\vee=\C[[s]],\qquad
J_m^{\vee\vee}=s^m\C[[s]].
\]
There is a length-one bidual defect, and $J_m$ is an explicit ideal realization of a canonical module of $B_m$. Its positive Ext groups into $B_m$ are nonzero in every degree and admit closed formulas. These computations explain why the lifting problem is naturally formulated with a canonical target. We also give two proper quotient stacks having the same coarse curve and the same ordinary coarse ideal but different complete linear systems upstairs.

\subsection{Deformations invisible to a fixed canonical fibre}
The next result compares a deformation of the quotient curve with the induced family of fixed canonical fibres.

\begin{maintheorem}[Invisible deformation and finite detection]\label{main:invisible}
Consider the marked equivariant family
\[
R_\beta=\C[[\beta,x,y]]/(y^6-x^7-\beta x^4y^4),
\qquad \mu_5:\ \wt(x,y)=(1,2).
\]
Its curve deformation is nontrivial, while its simultaneous normalization, conductor and coarse subring are constant under the normalization marking in \eqref{win-param}. For ordinary character ranks $(4,4,4,3,3)$, the entire scheme-valued fibre over $K_{\rm ss}=\C[[\beta,s]]\,ds\subseteq\omega_C$ is formally constant over $\C[[\beta]]$.

Nevertheless the character-two lattice is detected by a length-four jet. On the big cell of $\Gr(2,\C[s]/s^4)$ with echelon basis
\[
1+As^2+Bs^3,\qquad s+Cs^2+D_1s^3,
\]
the truncated unit action has the complete invariant $A-D_1+C^2$. On the family above its value is $\beta/6$.

The Cartier family $(x^3+\beta y^4/6)$ has constant canonical contraction and constant ordinary reflexive hull, but its ordinary coarse ideal varies and detects the same parameter.
\end{maintheorem}

Section~\ref{sec:invisible} proves the theorem using a sixteen-dimensional lattice window. A square-zero operator $T$ gives the polynomial gauge $1+\beta T$, which identifies the fixed-fibre functors to all orders even though the curve deformation is nontrivial. The jet invariant is defined relative to the fixed coarse-ring embedding and normalization parameter.

The compensation identity is especially concrete: $\Gamma_\omega(x^3-\alpha y^4)=K_{\alpha+\beta/6}$. The curve and Cartier-divisor motions cancel on the canonical target along $\alpha=-\beta/6$. An ordinary pushforward uses a different isotypic summand and does not exhibit the same cancellation. The comparison therefore depends essentially on which stabilizer character is retained.

\subsection{Conductor corrections after support restriction}
A second source of information loss appears on divisor-moduli spaces themselves. Let $A$ be a reduced affine chart, $N$ its finite normalization and $\mathfrak c=(A:N)$ its conductor. If $\mathfrak s$ is a support ideal, taking the normalization and conductor square first and then imposing $\mathfrak s$ need not reconstruct $A/\mathfrak s$. The actual error is
\begin{equation}\label{eq:intro-correction}
\mathcal K_{A,\mathfrak s}
=\frac{\mathfrak c\cap\mathfrak sN}{\mathfrak c\cap\mathfrak s}.
\end{equation}
This correction is generally smaller than $\ker(A/\mathfrak s\to N/\mathfrak sN)$. Within the framework of Ferrand's conductor square \cite{ferrand}, we determine it for an explicit equivariant divisor problem.

\begin{maintheorem}[A complete mixed-dimensional correction]\label{main:correction}
For the algebraic rank-three canonical-support chart of Section~\ref{sec:chart} at
\[
I_*=(x^6,x^5y,x^4y^2,x^2y^3,xy^4,y^5)
\subseteq\C[[x,y]]/(y^6-x^7)
\]
with ordinary character ranks $(4,4,4,3,3)$, the kernel in \eqref{eq:intro-correction} is
\[
\mathcal K\simeq
\bigl(\omega_{\mathcal T}\oplus\omega_{L_a}\oplus\omega_{L_b}\bigr)(-7),
\]
where
\[
\mathcal T=\Spec\frac{\C[a,b,d,f]}{(ab,bf,ad-af-df)},
\qquad \omega_{\mathcal T}\simeq(d,f).
\]
Its dimension filtration has no zero-dimensional layer, has the two line modules as its complete one-dimensional layer, and has $\omega_{\mathcal T}$ as its pure two-dimensional quotient, with the common grading shift indicated above. The actual square-zero extension of the reconstructed algebra is nonsplit. Completing all the displayed objects at $[I_*]$ gives the corresponding formal-local statement.
\end{maintheorem}

Section~\ref{sec:chart} constructs the source chart, identifies its seven reduced components, normalizes the remaining factor in two finite steps, and computes the conductor and support intersections. The proof combines geometric arguments for chart coverage and normality with exact polynomial calculations. We first determine the full kernel from explicit ideal intersections and then identify its canonical layers; the supplied code verifies the finite algebraic identities used in these steps.

The ideal $(d,f)$ is a canonical ideal on a $G_1+S_2$ surface. Since it is written as $\OO_{\mathcal T}(-D)$, its zero divisor is \emph{anticanonical}, in Hartshorne's convention. Its zero scheme is the node $\Spec\C[a,b]/(ab)$, whose two branches are precisely the lower-dimensional supports. The module splitting does not split the algebra extension, and its grading shift is retained throughout.

\subsection{Relation to the earlier preprint}\label{sec:earlier-preprint}
This article revisits the program begun in \emph{Generalized Divisors on DMH Stacks} \cite{dou}. All references to that preprint below concern version~1. We retain the terminology DMH for the $G_1+S_2$ condition, while focusing here on reduced tame curves and their finite-quotient families.

Example~3.5 of \cite{dou} does not give the claimed counterexample to \'etale compatibility of total quotient sheaves. In its notation, $A=k[t]$ and $x=t(t-1)$, so $t^{-1}=(t-1)x^{-1}$ and $A_{(t)}\otimes_{k[x]_{(x)}}k(x)=k(t)$. We withdraw that counterexample and the resulting obstruction asserted in \S3.1 and at the beginning of \S3.2. On locally Noetherian $S_1$ schemes, total quotient sheaves are compatible with \'etale pullback \cite[Tag 0EN6]{stacks}. The statements on reflexivity and duals in Propositions~3.6 and~3.9 of \cite{dou} remain valid by flat base change for Hom from finitely presented modules \cite[Tag 087R]{stacks}; this supplies a correct proof in place of the argument given for Proposition~3.6.

We also distinguish the abstract reflexive rank-one sheaves of Definition~3.10 in \cite{dou} from divisors with a specified rational trivialization. Generic rank one on \'etale charts does not by itself supply an equivariant rational trivialization when generic inertia is present. The effective ideal--substack correspondence of Theorem~3.13 is retained, with the inclusion part of the data. For unframed effective realizations, the quotient by global units in Proposition~3.16 must in general be replaced by the action of the full sheaf automorphism group. The unit quotient is valid when these groups agree, for example for an invertible sheaf. The groupoid formulation in Section~\ref{lat-sec:abel} keeps these automorphisms and the specified inclusions.

Theorems~\ref{main:lifting}--\ref{main:correction} are additions to \cite{dou}. Theorem~\ref{main:lifting} gives the exact lattice interval and the relative invariant-Quot construction under explicit hypotheses. Theorem~\ref{main:invisible} supplies a marked deformation invisible to the entire fixed canonical fibre, together with its finite-jet detection. Theorem~\ref{main:correction} computes a full normalization and the actual conductor correction, including its dimension filtration and nonsplit algebra extension. Their proofs are given in this article and do not use the withdrawn argument.

\subsection{Relation to existing work and organization}
Hartshorne's theory of generalized divisors and canonical reflexivity provides the $S_2$ framework \cite{hartshorne}. We use tameness and exact coarse pushforward in the sense of Abramovich--Olsson--Vistoli \cite{aov}, together with Nironi's stack duality and ordinary finite duality \cite{nironi,stacks}. The moduli constructions draw on the Quot results of Olsson--Starr and Hall--Rydh and on the multigraded Hilbert framework of Haiman--Sturmfels \cite{olssonstarr,hallrydh,hs}. Jiang's derived treatment of Abel maps gives a broader setting for these questions \cite{jiang}; the lifting constructions in this paper are classical.

The interpretation of nodal endomorphisms and automorphisms through partial normalization is developed in the compactified-Jacobian literature \cite[Lemmas 4.1--4.2]{cmkv}, and equisingular analytic moduli of plane branches have an established theory \cite{hefezhernandes}. Our deformation theorem tracks the coarse ring, stabilizer characters, canonical lattices and fixed fibres simultaneously. Related affine-paving results for torus-knot Quot spaces appear in \cite{hjo}; the fixed canonical-contraction conditions considered here are imposed by the scheme-theoretic equations below.

Section~\ref{sec:coarsening} develops canonical coarsening and the lattice-lifting theorem. Sections~\ref{sec:nongorenstein} and \ref{sec:length-one} study explicit quotient curves and classify their length-one canonical fibres. Section~\ref{sec:invisible} proves the deformation and jet-detection results. Section~\ref{sec:chart} constructs the rank-three chart and computes its conductor correction. Section~\ref{sec:second-order} compares the full source--support pair through second order, and Section~\ref{sec:perspectives} discusses the remaining questions. The appendices contain explicit generators, an equivalent incidence presentation and reproduction instructions.

%% file: sections/02_canonical_coarsening.tex
\section{Canonical coarsening and lattice fibres}\label{sec:coarsening}\label{sec:lattices}

\subsection{Divisors and the DMH condition}
All algebraic and formal-local constructions are over $\C$. Stacks are separated and locally Noetherian, with finite inertia, and all finite group orders are invertible. A \emph{DMH stack} in this article is a Deligne--Mumford stack satisfying $G_1$ and $S_2$ on an \'etale atlas. Here $G_1$ means Gorenstein at every point of codimension at most one. We work with reduced, pure-dimensional stacks. In pure dimension one, the DMH condition therefore says that the local curve rings on an atlas are Gorenstein. The term DMH serves only to designate this class of Deligne--Mumford stacks.

A generalized divisor is represented by a nondegenerate coherent fractional ideal together with its inclusion into the total quotient sheaf. On a $G_1+S_2$ stack we require this ideal to be reflexive on an \'etale atlas. An effective divisor is an actual inclusion $I\hookrightarrow\OO_{\mathcal X}$. In dimension one its quotient has zero-dimensional support, and torsion-free rank-one ideals are reflexive. Our definitions and conventions for ordinary and canonical reflexivity follow Hartshorne \cite{hartshorne}. The inclusion records the rational trivialization and hence the generic stabilizer character; an unmarked rank-one sheaf alone does not determine the divisor.

For example, a nontrivial character line on $B\mu_n$ is a line bundle but has no nonzero map to the trivial structure line. More generally, a rank-one sheaf with nontrivial generic inertia character cannot admit a generically nonzero inclusion into $\OO_{\mathcal X}$. Throughout, every ideal carries the linearization induced by its inclusion.

For a reduced curve, a \emph{full lattice} is a coherent submodule of the total quotient module which agrees with it at every generic point. If $I_D=\OO(-D)\simeq\omega_X$ is an ideal, we call $I_D$ a canonical ideal and $D$ an \emph{anticanonical} generalized divisor. We use this sign convention in Section~\ref{sec:conductor}.

\subsection{Finite quotient charts and relative duality}
We work over $\kk$. Our local model is a reduced one-dimensional Gorenstein algebra $R$, of finite type locally over $\kk$, or the completion of such a ring at a closed point, with an action of a finite group $G$; set $B=R^G$. The group may be noncyclic, and its action may exchange branches. Set $\Omega=\omega_B,\quad W=\omega_R$. We equip $W$ with the natural equivariant structure supplied by duality, without an additional character twist. The invariant ring is Cohen--Macaulay: locally a parameter of the reduced curve $B$ remains a nonzerodivisor on $R$. Finite duality identifies
\begin{equation}\label{lat-eq:finite-dual}
 W=\Hom_B(R,\Omega)
\end{equation}
compatibly with the $R$-action and the $G$-action, and its invariant part is $\Omega$. For a proper global quotient $[Y/G]$, these constructions apply sheafwise on the coarse curve. These are standard consequences of finite and tame duality; see \cite{nironi}.

For the relative Abel comparison in Section~\ref{lat-sec:abel}, we use the following setup. Let $S$ be locally Noetherian over $\kk$, let $Y/S$ be a flat projective family of geometrically reduced Gorenstein curves, and let the finite constant group $G$ act over $S$. Write $q:Y\to C=Y/G$, $g:C\to S$, $f=g\circ q:Y\to S$, $\XX=[Y/G]$, and $\pi:\XX\to C$. The relative dualizing sheaves come with compatible finite-duality identifications
\begin{equation}\label{lat-eq:relative-input}
 q_*\omega_{Y/S}=\mathcal Hom_C(q_*\OO_Y,\omega_{C/S}),
 \qquad \pi_*\omega_{\XX/S}=\omega_{C/S}.
\end{equation}
We require formation of these sheaves, their trace maps, and the displayed identifications to commute with the classical base changes under consideration. This is the usual relative Cohen--Macaulay duality package for finite quotients, and serves as the explicit duality input below. The coarse fibres are reduced Cohen--Macaulay curves, but need not be Gorenstein. For relative canonical modules and their compatibility with base change, see \cite[Tag 0E6N, Lemmas 109.19.1 and 109.19.3]{stacks} for curve families and \cite{jiang} for a derived treatment, including the discussion of Altman--Kleiman's theorem.

For $\Hom_B(R,K)$ with $K$ a varying \emph{noncanonical} lattice, formation need not commute with arbitrary base change; Section~\ref{lat-sec:jump} gives an example. Over an Artin base we work with relatively flat ideals and quotients, and the Gorenstein condition concerns the specified geometric fibres. The total space itself is not assumed to be an absolute DMH stack.

A \emph{canonical lattice} is a coherent inclusion $K\subseteq\Omega$ with finite-length quotient locally, or a quotient finite and flat over the base in a family. The inclusion is part of the data. On a non-Gorenstein coarse curve such a lattice need not be an $\OO$-reflexive generalized divisor. Upstairs, effective generalized divisors are ideals of finite colength; on a Gorenstein curve they are torsion-free and reflexive.

\subsection{The canonical Abel comparison}\label{lat-sec:abel}
Let $P_R=|G|^{-1}\sum_g g$ be the Reynolds map. In the model \eqref{lat-eq:finite-dual}, identify $\Omega$ with $W^G$ by
\begin{equation}\label{lat-eq:reynolds-inclusion}
 k\longmapsto\bigl(r\longmapsto P_R(r)k\bigr),\qquad
 W^G\longrightarrow\Omega,\quad \phi\longmapsto\phi(1).
\end{equation}
This specifies the normalization of the trace used below.

\begin{lemma}[Explicit adjunction]\label{lat-lem:adjunction}
For every finite equivariant $R$-module $N$ there is a natural isomorphism
\[
 \Hom_R(N,W)^G\simeq\Hom_B(N^G,\Omega).
\]
When $W$ is invertible, for a finite equivariant $F$ this becomes
\begin{equation}\label{lat-eq:correct-Hom}
 \Hom_R(F,R)^G\simeq
 \Hom_B\bigl((F\otimes_R W)^G,\Omega\bigr).
\end{equation}
These isomorphisms respect trace and evaluation. Under the relative input \eqref{lat-eq:relative-input}, they are compatible with the natural comparison maps for every specified classical base change. This compatibility is a naturality assertion; the Hom comparison maps need not be isomorphisms for an arbitrary $F$.
\end{lemma}
\begin{proof}
Given $\alpha:N^G\to\Omega$, the inverse adjunction sends $n$ to the functional $r\longmapsto \alpha(P_N(rn))$. This construction is $R$-linear in $n$ and equivariant. In the other direction, evaluate an equivariant map $N\to\Hom_B(R,\Omega)$ at $1$ on invariant elements. The constructions are inverse by $R$-linearity and averaging. Applying them to $N=F\otimes W$ and using the invertibility of $W$ gives the second isomorphism. The same formulas work sheaf-locally and are compatible with the base-change comparisons under \eqref{lat-eq:relative-input}.
\end{proof}

Let $\M_{\XX}$ be the groupoid of relatively flat rank-one torsion-free sheaves $F$ such that $\Gamma_\omega(F):=\pi_*(F\otimes\omega_{\XX/S})$ is also fibrewise rank-one torsion-free. Morphisms are all sheaf isomorphisms, with no simplicity condition or scalar rigidification. This generic condition is automatic for the effective ideals under discussion. Let $\M_C$ be the groupoid of relatively flat sheaves on $C$ whose geometric fibres are rank-one torsion-free.

Write $\Div_{\XX}$ for pairs $F\hookrightarrow\OO_{\XX}$ with flat quotient of relatively finite coarse support. On $C$, let $\Quot^{\rm lat}_{\omega_C}$ parametrize pairs $K\hookrightarrow\omega_{C/S}$ with the same flatness and finiteness conditions. Our Abel maps forget the inclusion and retain the source sheaf; this convention differs from the dual versions of the classical Abel map.

\begin{theorem}[Canonical Abel pullback]\label{lat-thm:abel}
In the relative setting above, the diagram
\[
\begin{array}{ccc}
\Div_{\XX} &\longrightarrow& \Quot^{\rm lat}_{\omega_C}\\
\downarrow &&\downarrow\\
\M_{\XX}&\xrightarrow{\ \Gamma_\omega\ }&\M_C
\end{array}
\]
is 2-Cartesian as a diagram of classical groupoid-valued functors. In particular, given $F$ and a specified coarse injection $\bar u:\Gamma_\omega(F)\hookrightarrow\omega_{C/S}$, there is a unique lift $u:F\hookrightarrow\OO_{\XX}$ inducing it. The quotients are flat. The statement requires no coarse Gorenstein, End-base-change, cohomology-bundle, or simplicity assumption.
\end{theorem}
\begin{proof}
Lemma~\ref{lat-lem:adjunction} gives a unique morphism $u$ corresponding to $\bar u$. On each geometric generic component, descent for the finite quotient reduces to a separable orbit of fields. The imposed rank condition makes both spaces of maps one-dimensional and identifies nonzero maps. Thus $u$ is generically nonzero on every component and is injective on each fibre because its source is torsion-free. The fibrewise flatness criterion gives injectivity on the family and a quotient flat over the base \cite[Tag 00MD, Lemma 10.99.1]{stacks}. The quotient has proper support that is finite on fibres, hence relatively finite coarse support. Conversely, tensoring an ideal inclusion with the invertible relative canonical sheaf and applying exact invariants gives a canonical-lattice inclusion with flat quotient. Uniqueness and naturality establish the assertion on morphisms and its compatibility with repeated base change.
\end{proof}

An object of the fibre product includes a specified \emph{isomorphism} between $\Gamma_\omega(F)$ and the source of the coarse injection; equality of sheaf classes alone is insufficient. The target sheaf stacks retain $\Aut(F)$ and $\Aut(K)$, whereas an injective pair has trivial stabilizer because $u\alpha=u$ forces $\alpha=1$. On bounded charts, standard Coh and Quot theory supplies the corresponding classical algebraic objects \cite{hallrydh}. The pullback statement does not assert that $\Div_{\XX}\to\Quot^{\rm lat}_{\omega_C}$ is an equivalence, surjective in each fixed degree, or flat: its fibres still allow $F$ itself to vary.

\begin{remark}[The change of target]
For a torsion-free module $K$ on a one-dimensional Cohen--Macaulay ring,
$\Ext^{>0}_B(K,\Omega)=0$ and canonical biduality holds. Thus the local framed directions into $\Omega$ have no positive Ext, unlike directions into $B$. Canonical-module duality accounts for this difference from the unbounded $\Ext_B(J,B)$ in Theorem~\ref{ng-thm:ext}. The comparison changes both the source to $\Gamma_\omega(F)$ and the target to $\Omega$, so it does not imply unobstructedness of the original ordinary-pushforward problem.
\end{remark}

\subsection{All lifts of a fixed canonical lattice}\label{lat-sec:interval}
We now fix a local quotient over $\kk$ and view $\Omega\subseteq W$ through \eqref{lat-eq:reynolds-inclusion}. For a full lattice with a specified inclusion $K\subseteq\Omega$, define
\begin{equation}\label{lat-eq:endpoints}
 \Lambda_K^-:=RK\subseteq W,\qquad
 \Lambda_K^+:=\Hom_B(R,K)\subseteq\Hom_B(R,\Omega)=W.
\end{equation}
The lower endpoint is the image of induction; we do not require the natural map from $R\otimes_B K$ to $W$ to be injective.

\begin{theorem}[The exact lattice interval]\label{lat-thm:interval}
Both endpoints in \eqref{lat-eq:endpoints} are coherent $G$-stable full $R$-lattices in $W$. They satisfy $(\Lambda_K^-)^G=(\Lambda_K^+)^G=K$. The $G$-stable $R$-submodules $L\subseteq W$ with $L^G=K$ are exactly
\begin{equation}\label{lat-eq:interval}
                 \Lambda_K^-\subseteq L\subseteq\Lambda_K^+.
\end{equation}
Consequently the ideal lifts of $K$ are exactly $I=L\otimes_R W^{-1}\subseteq R$, with $L$ as in \eqref{lat-eq:interval}. In particular, every canonical lattice has an ideal lift over the field when the upstairs quotient length and representation are allowed to vary.
The finite module
\begin{equation}\label{lat-eq:window}
                    \Delta_K:=\Lambda_K^+/\Lambda_K^-
\end{equation}
has no invariants: $\Delta_K^G=0$.
\end{theorem}
\begin{proof}
If $\sum r_i k_i\in RK$ is invariant, averaging it gives $\sum P_R(r_i)k_i\in K$. This proves $(RK)^G=K$. Similarly, evaluation at $1$ identifies the invariant part of $\Hom_B(R,K)$ with $K$ by the same Reynolds formula.

If $L^G=K$, the submodule contains $K$ and hence $RK$. For $\phi\in L\subseteq W$ and $r\in R$, the element $r\phi$ lies in $L$. Its Reynolds average lies in $K$, and evaluation at $1$ is $\phi(r)$. Thus $\phi(R)\subseteq K$, proving $L\subseteq\Hom_B(R,K)$. Conversely, every intermediate $G$-stable lattice has invariant part $K$, since the endpoints have the same invariant part.

Coherence follows because $R$ is finite over $B$. At each generic point, the finite characteristic-zero quotient of a product of fields is separable on each orbit. The invariant canonical line spans the canonical module after extension to those fields. Thus $RK$, and hence both endpoints, have full rank on every branch. The resulting quotients have zero-dimensional support and finite local length. Invertibility of $W$ converts these sublattices into actual ideals. Exactness of invariants applied to the endpoint sequence gives $\Delta_K^G=0$.
\end{proof}

The upper endpoint also has a description independent of a choice of basis of $W$:
\begin{equation}\label{lat-eq:max-test}
 \Lambda_K^+=\{w\in W:P_W(rw)\in K\text{ for every }r\in R\}.
\end{equation}
This condition tests all multiples of $w$. Once $K$ is fixed, the remaining choice is a submodule of $\Delta_K$ compatible with its residual $R$-action and $G$-action. Isotypic dimensions alone do not record these multiplication conditions.

For a proper global quotient, the endpoint construction uses the finite algebra $q_*\OO_Y$ on $C$ and glues through finite-duality evaluation. The window is then an equivariant sheaf with finite support. A change of quotient presentation requires transporting the trace and the canonical module together; unrelated character trivializations would lose this identification.

\subsubsection{Projectivity with fixed representation}
Fix the representation of the finite quotient $W/L$. Equivalently, in a local semilinear trivialization of $W$ of character $\theta$, fix the representation of $R/I$ and then tensor it by $\theta$.

\begin{corollary}\label{lat-cor:grass}
The fixed-$K$, fixed-representation fibre is the scheme of $G$-stable $R$-submodules $N\subseteq\Delta_K$ with the prescribed isotypic dimensions, under $N=L/\Lambda_K^-$. It is a closed subscheme of a finite product of ordinary Grassmannians and is therefore projective. It may be empty for a specified representation. Allowing all possible representations gives a finite union of such schemes.
\end{corollary}
\begin{proof}
The image of $R$ in $\End_\kk(\Delta_K)$ is a finite-dimensional algebra. Adjoin the finite $G$-action, or work with the corresponding finite skew-group algebra. Decompose
$\Delta_K=\bigoplus_V V\otimes D_V$ into irreducible $G$-representations. Subbundles of fixed $G$-type are parameterized by $\prod_V\Gr(d_V,D_V)$. For each of finitely many algebra generators, impose the vanishing of $N\longrightarrow\Delta_K\longrightarrow\Delta_K/N$, where the first arrow is the specified generator action. These closed equations express exactly the stability of the submodule, including over nonreduced test algebras.

Taking the inverse image of $N$ gives $L$. Its quotient in $W$ is an extension of the fixed finite module $W/\Lambda_K^+$ by $\Delta_K/N$, so it is flat over the test base. Conversely a flat quotient $W/L$ gives a flat quotient $\Delta_K/N$ by the same exact sequence. These constructions are inverse and commute with scalar extension from $\kk$.
\end{proof}

This is the usual submodule-scheme construction; see \cite{caldero,brion}. Its tangent space at $N$ is $T_N=\Hom_{R\rtimes G}(N,\Delta_K/N)$. The construction reduces each specified canonical-coarsening fibre to finite matrices, without imposing smoothness, reducedness, or rationality on the resulting scheme.

\subsection{A relative Quot presentation}\label{lat-sec:relative-window}
The interval over a fixed field is useful for computation. For families, where coinduction need not commute with base change, we use a presentation by coherent quotients. Keep the proper finite-quotient family of Section~\ref{sec:coarsening}, and choose a canonical lattice $K\subseteq\omega_{C/S}$ with finite flat quotient. On $Y$ let
\begin{equation}\label{lat-eq:induced-quotient}
 \mathcal T_K=
 \coker\bigl(q^*K\longrightarrow\omega_{Y/S}\bigr).
\end{equation}
The map is induced by the normalized canonical trace inclusion. Taking its \emph{cokernel} makes the construction applicable even when its image is not a subbundle.

\begin{theorem}[Relative invariant-Quot presentation]\label{lat-thm:relative-quot}
The coherent equivariant sheaf $\mathcal T_K$ has support finite over $S$. Its formation commutes with arbitrary classical base change, without requiring that it be $S$-flat. There is a canonical identification
\begin{equation}\label{lat-eq:trivial-quot}
                    (q_*\mathcal T_K)^G=\omega_{C/S}/K.
\end{equation}
Fix multiplicities $h(V)$ for the irreducible representations of $G$ in the finite locally free relative pushforward of an upstairs canonical quotient along $f:Y\to S$. Write $d=\operatorname{rk}g_*(\omega_{C/S}/K)$ and require $h(\mathbf1)=d$. For an $S$-scheme $T$, write $K_T$ for the pullback lattice. Then $G$-stable ideal families $I\subseteq\OO_{Y_T}$ satisfying
\[
 ((q_T)_*(I\otimes\omega_{Y_T/T}))^G=K_T,
 \qquad (f_T)_*\bigl((\OO_{Y_T}/I)\otimes\omega_{Y_T/T}\bigr)\text{ has type }h
\]
are equivalent to $G$-equivariant, $T$-flat quotients of $(\mathcal T_K)_T$ of type $h$. This is a projective relative Quot scheme (with the prescribed representation components), compatible with base change. The construction requires neither End-base-change assumptions nor flatness or base-change compatibility of the upper envelope.
\end{theorem}
\begin{proof}
The base-change assertion follows from right exactness of pullback and the prescribed compatibility of the canonical sheaves and trace with base change. Fibrewise the invariant canonical line spans $\omega_Y$ at every generic component. Thus $\mathcal T_K$ has zero-dimensional fibre support; its proper support over $S$ is finite. Applying exact invariants to the defining cokernel gives \eqref{lat-eq:trivial-quot}, by the Reynolds calculation $(RK)^G=K$. The finite pushforward of $\mathcal T_K$ is coherent and commutes with base change even without flatness.

An ideal as in the statement gives a quotient $Q=(\OO/I)\otimes\omega_Y$. Its kernel contains the image of $q^*K$, so $Q$ is a quotient of $\mathcal T_K$. Conversely, take a flat equivariant quotient $\mathcal T_K\twoheadrightarrow Q$. Its composite with $\omega_Y$ has kernel $L$, and tensoring $L$ with the inverse canonical line gives an ideal $I\subseteq\OO_Y$ with flat quotient. Exactness of invariants gives a surjection $\omega_C/K\twoheadrightarrow(q_*Q)^G$. After finite pushforward to $S$, both terms are vector bundles of rank $d$: the rank on the right is the prescribed trivial multiplicity $h(\mathbf1)$. A surjection between equal-rank vector bundles is an isomorphism. Finite pushforward is faithful on these finite-support sheaves, so the sheaf map is an isomorphism as well. Hence the contracted canonical lattice is exactly $K$.

Quotients of a coherent sheaf on the projective $Y/S$ with fixed finite length form the usual projective Quot scheme. The finite $G$-invariance condition is closed; a stable kernel gives the quotient its unique induced equivariant structure. The representation multiplicities in the finite flat pushforward of the quotient are locally constant, so prescribing them selects an open-and-closed union of components. This gives the claimed projective scheme. All constructions commute with the stipulated base changes. For the general Quot and invariant-Hilbert frameworks, see \cite{hallrydh,brion}; projectivity in the present setting follows from the projective Quot construction above. Equivalently, finite support reduces the problem locally on $S$ to locally free quotients of a finitely presented module, with closed equations for multiplication and $G$-stability; no flatness of the input module is needed.
\end{proof}

The coherent input $\mathcal T_K$ may be nonflat, as it is in the example below. Projectivity alone does not imply flatness or surjectivity in every representation, and the invariant length alone does not determine the other representation multiplicities.

When the endpoints satisfy suitable flatness and base-change conditions, the smaller interval model is also available:
\begin{proposition}[The vector-window specialization]\label{lat-prop:relative-window}
Suppose the two endpoint constructions commute with the specified base changes, and both quotients $W/\Lambda_K^-$ and $W/\Lambda_K^+$ are finite flat over the base. Then $\Delta_K$ is finite locally free after pushing to the base, and each fixed-representation fibre family is the corresponding relative closed submodule Grassmannian. It is projective and has the indicated base-change compatibility.
\end{proposition}
\begin{proof}
The kernel of $W/\Lambda_K^-\to W/\Lambda_K^+$ is $\Delta_K$, hence is flat and finite over the base. Repeat the proof of Corollary~\ref{lat-cor:grass} with vector bundles and finitely many relative multiplication maps. The flat endpoint quotients ensure that subbundles correspond precisely to flat quotient families. Compatibility of the endpoints with base change identifies this family with the window recomputed on each base change.
\end{proof}
This vector-window model gives a smaller presentation under the stated hypotheses. When they fail, the coherent-quotient model of Theorem~\ref{lat-thm:relative-quot} still applies, as the explicit calculation in Section~\ref{lat-sec:jump} shows.

%% file: sections/03_ordinary_coarsening.tex
\section{Ordinary coarsening and reflexivity}\label{sec:nongorenstein}

\subsection{A family of plane-branch quotients}
Throughout this section $k=\C$. We begin with an explicit family of plane-branch quotients; the more general existence of symmetric semigroup covers is known \cite{swanson}. In each example we keep track of the specified ideal upstairs.
Fix $m\geq3$, and put
\[
 n=2m-3,\qquad a=n+1=2m-2,\qquad
 R_m=k[[t^a,t^{a+1}]],\qquad s=t^n,\qquad N=k[[s]].
\]
Let $\mu_n$ act through $t\mapsto\zeta t$. It preserves $R_m$, and acts on its two
plane generators with weights $1,2$ modulo $n$. Since $\gcd(a,a+1)=1$, $R_m\simeq k[[x,y]]/(y^a-x^{a+1})$ is a one-dimensional Gorenstein domain. Write $\mathcal X_m=[\Spec R_m/\mu_n]$ for the
local stack (or use the corresponding algebraic local model and complete at the cusp).

\begin{lemma}[Exact membership test]
For $a\geq2$ and an integer $v\geq0$, write $v=qa+r$ with $0\leq r<a$. Then
\[
 v\in\langle a,a+1\rangle\quad\Longleftrightarrow\quad q\geq r.
\]
The conductor of this semigroup is $a(a-1)$.
\end{lemma}
\begin{proof}
If $q\geq r$, write $v=(q-r)a+r(a+1)$. Conversely, in a representation
$v=ia+j(a+1)$, reduce $j$ modulo $a$; its least possible nonnegative value is $r$,
which forces $q\geq r$. The criterion gives the asserted conductor, with
$a(a-1)-1$ as the last missing integer.
\end{proof}

\begin{theorem}[Coarse ring and its type]\label{ng-thm:ring}
The invariant ring is
\[
 B_m=R_m^{\mu_n}=k+s^mN
     =k[[s^m,s^{m+1},\ldots,s^{2m-1}]].
\]
Its Cohen--Macaulay type is $m-1$. In particular, the coarse curve is not Gorenstein,
and the type is unbounded in this family.
\end{theorem}
\begin{proof}
An invariant monomial is $t^{nL}$. The membership lemma shows that its positive
$s$-exponents are exactly $L\geq m$. Indeed, for
$1\leq L<a$, the Euclidean division of $nL=aL-L$ has quotient $L-1$ and remainder
$a-L$, so the membership condition is $2L\geq a+1$, equivalently $L\geq m$. The conductor gives membership at the boundary $L=a$
and at every larger exponent. Invariants of formal series are
obtained coefficientwise, proving the ring equality.

The ring is one-dimensional Cohen--Macaulay, since it is a finite domain over the DVR
$k[[s^m]]$. The Artinian quotient by the parameter $s^m$ has basis $1,s^{m+1},\ldots,s^{2m-1}$. Its nonconstant basis elements form its socle: their products with the maximal ideal vanish,
whereas $1$ is not in the socle. The type is therefore $m-1$.
\end{proof}

\subsection{Cartier upstairs, nonreflexive downstairs}
Consider the effective Cartier ideal $I_m=t^aR_m\subset R_m$, with the linearization induced by this inclusion, and put $J_m=I_m^{\mu_n}$.

\begin{theorem}[Duals and endomorphisms of the coarse ideal]\label{ng-thm:ideal}
Inside $k((s))$ one has
\begin{align}
 J_m&=(s^m,s^{m+1},\ldots,s^{2m-2})B_m,\label{ng-eq:J}\\
 J_m^\vee&=N,\qquad
 J_m^{\vee\vee}=s^mN,\qquad
 \length_{B_m}(J_m^{\vee\vee}/J_m)=1,\label{ng-eq:bidual}\\
 \End_{B_m}(J_m)&=B_m,\label{ng-eq:End}\\
 (I_m^\vee)^{\mu_n}&=k+s^{m-1}N\subset N=J_m^\vee.\label{ng-eq:invdual}
\end{align}
Thus the dual-comparison cokernel has length $m-2$, although $I_m$ is Cartier.
The pushforward $J_m$ is not reflexive.
\end{theorem}
\begin{proof}
The exponents of $J_m$ are the integers $L$ for which $nL-a$ belongs to
$\langle a,a+1\rangle$. The membership test gives
\[
 v(J_m)=\{m,m+1,\ldots,2m-2\}\cup\{2m,2m+1,\ldots\}.
\]
For $1\leq L<a$, the quotient and remainder for $nL-a$ are $L-2$ and $a-L$.
The resulting condition is $L\geq m$. Membership also holds at $L=a$, but fails at $L=a+1=2m-1$. For $L\geq2m$ the conductor applies. This proves \eqref{ng-eq:J}.

For $J_m^\vee=(B_m:J_m)$, every nonnegative power of $s$ is an allowed multiplier.
A negative exponent $e$ fails on $s^m$, except possibly $e=-m$; that exponent fails
on $s^{m+1}$. More negative exponents also fail. Hence $J_m^\vee=N$.
The conductor $(B_m:N)$ is $s^mN$, proving \eqref{ng-eq:bidual}. The only missing
monomial of $J_m$ inside $s^mN$ is $s^{2m-1}$.

An endomorphism multiplier cannot have negative valuation. If $1\leq e<m$, choose
$g=2m-1-e\in[m,2m-2]$; then $s^e s^g=s^{2m-1}\notin J_m$. Exponents zero and
at least $m$ preserve $J_m$. Applying these tests coefficientwise proves \eqref{ng-eq:End}.

Finally, an invariant inverse multiplier has exponent $L$ satisfying $nL+a\in
\langle a,a+1\rangle$. The quotient--remainder test, including the negative and
conductor ranges, gives exactly $L=0$ and $L\geq m-1$. Thus \eqref{ng-eq:invdual}
holds, and the missing exponents $1,\ldots,m-2$ give the stated length.
\end{proof}

\begin{proposition}[The coarse ideal as a canonical module]\label{ng-prop:canonical}
A fractional canonical module of $B_m$ is
\[
 \Omega_m=(1,s,\ldots,s^{m-2})B_m
 =\{\text{series with exponents in }[0,m-2]\cup[m,\infty)\}.
\]
Consequently $J_m=s^m\Omega_m\simeq\omega_{B_m}$ as a $B_m$-module.
\end{proposition}
\begin{proof}
Set $D=k[[s^m]]$. The $D$-basis of $B_m$ is
$1,s^{m+1},\ldots,s^{2m-1}$. Finite duality identifies a canonical module with
$\Hom_D(B_m,D)$. The field trace $k((s))/k((s^m))$ realizes this dual as a
fractional lattice. The trace of $s^q$ vanishes unless $m\mid q$, in which case it is
$m s^q$. Thus the least exponent of a trace-dual monomial in residue class zero
modulo $m$ is $0$, and in class $r$, $1\leq r<m$, it is $r-2m$.
Multiplication by $s^{2m-1}$ identifies this lattice with $\Omega_m$.
Equation \eqref{ng-eq:J} gives $J_m=s^m\Omega_m$.
\end{proof}

\begin{remark}[The role of the canonical module]
The coarse pushforward $J_m$ of an upstairs Cartier ideal is thus a canonical module embedded as an ideal on the non-Gorenstein coarse curve. Passing to $J_m^{\vee\vee}$ replaces the effective ideal and changes the quotient length by one.
\end{remark}

\subsection{Higher Ext groups}
We next compute Ext groups with the structure sheaf as target. These differ from the self-Ext groups that control deformations of the underlying sheaf.

\begin{theorem}[Exact higher-Ext growth]\label{ng-thm:ext}
For $m\geq3$,
\[
 \dim_k\Ext^1_{B_m}(J_m,B_m)=(m-1)(m-2)-1,
\]
and, for every $i\geq2$,
\[
 \dim_k\Ext^i_{B_m}(J_m,B_m)
       =m(m-2)^2(m-1)^{i-2}.
\]
All these groups are nonzero and have finite length. By contrast,
$\Ext^i_{R_m}(I_m,R_m)=0$ for every $i>0$.
\end{theorem}
\begin{proof}
We suppress the subscript $m$. The normalization $N$ is generated over $B$ by
$1,s,\ldots,s^{m-1}$. For $1\leq k\leq m-1$, consider the row
\[
 (1,s,\ldots,s^k):B^{k+1}\longrightarrow N.
\]
In any relation $\sum_{j=0}^k s^j b_j=0$, the distinct coefficients of degrees
$0,\ldots,k$ force every constant term of $b_j$ to vanish. Therefore
$b_j=s^m f_j$ for $f_j\in N$, and $f_0=-\sum_{j=1}^k s^j f_j$. This describes the entire kernel and identifies it with $N^k$. Multiplying the row by
$s^m$ and taking $k=m-2$ presents $J$. Taking $k=m-1$ presents $N$. Hence
\begin{equation}\label{ng-eq:syzygies}
 \Om_BJ\simeq N^{m-2},\qquad \Om_BN\simeq N^{m-1}.
\end{equation}
This computes the kernel over the full completed ring, without a finite-order approximation.

Since $\Hom_B(N,B)=s^mN$, dualizing the two presentations gives the relevant cokernel as that of the map
\[
 B^{k+1}\longrightarrow N^k,\qquad
 (b_0,\ldots,b_k)\longmapsto
 (b_1-sb_0,b_2-s^2b_0,\ldots,b_k-s^kb_0).
\]
The image contains $B^k$. Modulo $B^k$, the target $(N/B)^k$ has dimension
$k(m-1)$. Only the constant term of $b_0$ contributes to its image, producing the
nonzero vector $(s,s^2,\ldots,s^k)$; that image has dimension one. It follows that
\[
 \dim\Ext^1_B(J,B)=(m-2)(m-1)-1,\qquad
 \dim\Ext^1_B(N,B)=(m-1)^2-1=m(m-2).
\]
Dimension shifting along \eqref{ng-eq:syzygies} gives the formula for every $i\geq2$.
The last assertion follows because $I_m$ is free of rank one over $R_m$.
\end{proof}

For $m=3$ the successive dimensions are $1,3,6,12,24,48,\ldots$. They measure ordinary local Ext into the structure ring and should not all be interpreted as tangent dimensions. Moreover, tangent-space data alone do not decide smoothness or singularity of a classical moduli space.

\begin{corollary}[An unbounded ordinary derived comparison]\label{ng-cor:amplitude}
For the proper Cartier examples below, the ordinary comparison cone
\begin{equation}\label{ng-eq:comparison-cone}
\mathcal C_I=\operatorname{Cone}\!\left(
\RHom_{\mathcal X}(I,\OO_{\mathcal X})
\longrightarrow\RHom_C(\pi_*I,\OO_C)\right)
\end{equation}
has nonzero cohomology in arbitrarily high degrees. The arrow is induced by exact coarse pushforward. Consequently, a comparison retaining this full derived-Hom correction cannot replace it by a bounded perfect complex. The assertion concerns the full derived mapping problem, not the cotangent complexes of classical realization-space maps obtained by truncation.
\end{corollary}
\begin{proof}
At the cusp, the coarse ideals in the proper examples are fractional rescalings of $J_3$, so Theorem~\ref{ng-thm:ext} applies. Their positive sheaf Ext groups are supported at that point. The local-to-global spectral sequence on a curve therefore identifies global Ext with these nonzero local groups in every degree at least two. Upstairs $I$ is a line bundle, so $\RHom_{\mathcal X}(I,\OO_{\mathcal X})$ has cohomology only in global degrees zero and one. The long exact sequence of \eqref{ng-eq:comparison-cone} now gives $H^i(\mathcal C_I)\simeq\Ext_C^i(\pi_*I,\OO_C)\ne0$ for every $i\ge2$.
\end{proof}

\begin{remark}
The obstruction concerns the ordinary target $\OO_C$. It does not preclude an almost-perfect derived construction or the classical canonical-target comparison of Section~\ref{sec:coarsening}.
\end{remark}

\subsection{Different linear systems over a common coarse pair}
The loss of information also appears in global sections. We now construct two proper stacks with a common coarse pair and different complete linear systems.
For $a=4,5$, let $Y_a=V(Y^aZ-X^{a+1})\subset\mathbb P^2$. Its normalization has affine parameter $t$, with $X=t^a,Y=t^{a+1}$ on $Z=1$.
The only singular point is the cusp at $t=0$; the point at infinity is smooth, since
$\partial(Y^aZ-X^{a+1})/\partial Z=1$ at $[0:1:0]$.

Set
\[
 \mathcal X_1=[Y_4/\mu_3],\qquad
 \mathcal X_2=[Y_5/\mu_4],
\]
where in each case the normalization parameter is multiplied by the corresponding root
of unity. Both are proper tame Gorenstein DM curves. Write the second normalization
parameter as $z$, and identify $s=t^3=z^4$ on the two coarse normalizations.

\begin{proposition}[A common coarse curve]
Both coarse curves are the proper integral curve $C$ obtained by gluing $\Spec k[s^3,s^4,s^5]\quad\text{and}\quad\Spec k[s^{-1}]$ along $\Spec k[s,s^{-1}]$. Thus the two coarse curves agree as schemes, including before completion at the cusp.
\end{proposition}
\begin{proof}
The first affine invariant semigroup is given by the $m=3$ case of Theorem \ref{ng-thm:ring}.
For the second, the same Euclidean test on $4L\in\langle5,6\rangle$ gives exactly
$L=0$ and $L\geq3$. At infinity each invariant ring is $k[s^{-1}]$. The overlap is
unchanged, proving the gluing assertion. Equivalently, $C$ is the monomial projective
curve parametrized by $[S^5:S^2T^3:ST^4:T^5]$.
\end{proof}

Define Cartier ideals supported only at the cusp by local equations
\[
 I_1=t^{16}k[[t^4,t^5]],\qquad
 I_2=z^{17}k[[z^5,z^6]],
\]
and by the structure ideal away from the cusp. These glue as effective invertible
ideal sheaves: the local equations belong to their rings and are units on the punctured
cusp chart. The ideals are stable and carry the linearizations induced by these inclusions.

\begin{theorem}[Distinct realizations over the same coarse ideal]
The two pushforward ideal sheaves coincide. Their common lattice at the cusp is $J=(s^7,s^8)B,\quad B=k[[s^3,s^4,s^5]]$, while away from the cusp $J=\mathcal O_C$. The local duals satisfy
\begin{align*}
 J^\vee&=s^{-4}k[[s]],\\
 \pi_{1*}I_1^\vee&=s^{-4}(k+s^2k[[s]]),\\
 \pi_{2*}I_2^\vee&=s^{-3}B.
\end{align*}
Their complete effective-realization spaces map to the same coarse realization space as
\[
 \mathbb P^3\hookrightarrow\mathbb P^4,
 \qquad
 \mathbb P^1\hookrightarrow\mathbb P^4,
\]
respectively. With coarse section basis $(1,s^{-1},s^{-2},s^{-3},s^{-4})$, the first image
is the hyperplane with coefficient of $s^{-3}$ equal to zero; the second is the line
spanned by $1,s^{-3}$.
\end{theorem}
\begin{proof}
Membership in $\langle4,5\rangle$ and $\langle5,6\rangle$ gives respectively
\[
 3L-16\in\langle4,5\rangle,
 \qquad 4L-17\in\langle5,6\rangle.
\]
Both conditions yield $L=7,8$ or $L\geq10$, which proves the assertion about the common ideal. Since it
is $s^4$ times $J_3$ from Theorem \ref{ng-thm:ideal}, its dual is $s^{-4}k[[s]]$.
The inverse invariant conditions are
\[
 3L+16\in\langle4,5\rangle,
 \qquad 4L+17\in\langle5,6\rangle.
\]
They give $L=-4$ or $L\geq-2$, and $L=-3$ or $L\geq0$, respectively. These assertions follow either by checking the finitely many exponents below the conductor bounds or directly from the exact Euclidean test.

At infinity all three lattices equal the structure sheaf. Hence the global section
spaces are spanned by the monomials with nonpositive exponents in the displayed cusp lattices:
\[
 H^0(J^\vee)=\langle1,s^{-1},s^{-2},s^{-3},s^{-4}\rangle,
\]
\[
 H^0(I_1^\vee)=\langle1,s^{-1},s^{-2},s^{-4}\rangle,
 \qquad H^0(I_2^\vee)=\langle1,s^{-3}\rangle.
\]
The curves are integral, so every nonzero section is generically nonzero. Upstairs the
ideals are invertible and the proper connected curve has only constant functions.
Downstairs, Theorem \ref{ng-thm:ideal} gives $\End(J)=\mathcal O_C$ globally as well. Thus in
all three cases the automorphism group is $k^\times$, and projectivizing these explicit
section inclusions gives the statement.
\end{proof}

\begin{remark}[Terminology on the coarse curve]
The common coarse ideal $J$ is not $\mathcal O_C$-reflexive, and $C$ is not Gorenstein.
Accordingly, the $\mathbb P^4$ above parametrizes effective ideal images isomorphic to a specified rank-one torsion-free sheaf, rather than a reflexive generalized-divisor system on a DMH coarse curve. The two dual-comparison lengths are $1$ and $3$, even though the coarse pair $(C,J)$ is fixed.
\end{remark}

%% file: sections/04_length_one.tex
\section{Length-one canonical quotients}\label{sec:length-one}
\subsection{The punctual moduli problem}
We first describe the canonical quotients on the coarse curve, then specify the equivariant ideal families that lift them.

\subsubsection{Plane branches and canonical generators}
Set
\[
 r=m-2\ge1,\qquad n=2r+1=2m-3,\qquad a=n+1=2r+2,
\]
\[
 R_m=\C[[x,y]]/(y^a-x^{a+1}),\qquad
 \mu_n:\ (x,y)\longmapsto(\zeta x,\zeta^2y).
\]
The normalization is $\C[[t]]$, with $x=t^a$, $y=t^{a+1}$. These are
Gorenstein plane branches, so their tame quotient stacks are examples of
the DMH curves under discussion. The coarse ring is
\[
 B_m=R_m^{\mu_n}=\C+s^m\C[[s]],\qquad s=t^n.
\]
Recall that $v=qa+b$, $0\le b<a$, belongs to
$\langle a,a+1\rangle$ exactly when $q\ge b$. Applying this to $v=nL$
shows that the invariant exponents are $L=0$ and $L\ge m$.
The quotient $B_m/(s^m)$ has socle dimension $m-1$; in particular the
coarse curve is not Gorenstein for the parameters considered here.

We use the residue generator
\[
 \eta=\frac{dx}{a y^{a-1}}=t^{-a(a-1)}dt,
 \qquad \wt(\eta)=1\pmod n.
\]
Write $\Omega_m=(R_m\eta)^{\mu_n}=\omega_{B_m}$, with this natural
linearization. The elements
\begin{equation}\label{one-eq:canonical-generators}
 h_j=x^{2j}y^{r-j}\eta,\qquad 0\le j\le r,
\end{equation}
are minimal $B_m$-generators of $\Omega_m$ and give a basis of
$\Omega_m/\m_{B_m}\Omega_m$. Indeed, in terms of $s$, they are
$n^{-1}s^{j-m}ds$; the fractional canonical module is a common shift of
$(1,s,\ldots,s^r)B_m$. This is also the trace-dual description in Proposition~\ref{ng-prop:canonical}.

Thus the space of length-one canonical quotients killed by
$\m_{B_m}$ is $\PP^r$, with homogeneous coordinates
$[T_0:\cdots:T_r]$ recording the values on the generators
\eqref{one-eq:canonical-generators}. Here projective space parametrizes one-dimensional quotients. On $T_0\ne0$, write these values as $(1,\tau_1,\ldots,\tau_r),\quad K_{\boldsymbol\tau} =\ker(\Omega_m\to\C)$. Over an arbitrary test algebra, we use the same notation with a locally
free quotient line in place of $\C$.

\subsubsection{Ideal families and the prescribed representation}
We fix the ordinary quotient representation
\begin{equation}\label{one-eq:rep}
 \rho_m=\C[\mu_n]\oplus\mathbf1.
\end{equation}
For a test scheme, a family in our moduli problem is a $\mu_n$-stable
ideal $I\subset R_m$ with a finite locally free quotient $Q=R_m/I$,
punctually supported at the chosen cusp, whose isotypic ranks are $\rk Q_0=2,\quad \rk Q_i=1\quad(1\le i\le n-1)$. We impose the \emph{actual canonical contraction}
\begin{equation}\label{one-eq:contraction}
 (I\eta)^{\mu_n}=K.
\end{equation}
The equality fixes the submodule itself, including its inclusion. The canonical
quotient is $Q_{n-1}\eta$, not $Q_0$; the twist by $\eta$ is essential.
Ideals are not identified under automorphisms of the ambient curve or
abstract isomorphisms of the quotient algebra.

Equivalently, $Q_{n-1}\eta$ is the specified length-one quotient of
$\Omega_m$, killed by $\m_{B_m}$. The action of $\mu_n$ on every test base is trivial. The uniform nilpotence bound proved below shows that the functor is represented on a fixed finite infinitesimal neighbourhood of the cusp, including over nonreduced test schemes. In a proper realization, we require the structure ideal everywhere away from the chosen cusp, including at every other stacky point.

Theorem~\ref{lat-thm:relative-quot} supplies the abstract invariant-Quot presentation. We now determine this scheme by explicit equations.

\subsection{An explicit finite window}\label{lat-sec:uniform}
For the same plane branch and action, recall the residue generator
\[
 \eta=\frac{dx}{a y^{a-1}}=t^{-a(a-1)}dt.
\]
It has character $1$ modulo $n$: $dx$ has weight $1$ and $y^{a-1}$ has weight $2n=0$. Thus $W=R_m\eta$. Fix $K_m=(xR_m\eta)^{\mu_n}\subseteq\Omega_m$. Since $R_m/(x)=\kk[y]/(y^a)$, the invariant part after tensoring by $\eta$ is the unique line generated by $y^{m-2}\eta$. Indeed $2j+1=0\pmod n$ has the unique solution $j=m-2$ for $0\le j\le a-1=n$. Therefore
\begin{equation}\label{lat-eq:canon-lengthone}
                       \length(\Omega_m/K_m)=1.
\end{equation}

\begin{theorem}[Explicit envelopes for every $m$]\label{lat-thm:uniform}
Write $\Lambda_{K_m}^{\pm}=I_m^{\pm}\eta$. Then
\begin{align}
 I_m^+&=(x,y^{m-1}),\label{lat-eq:plus-m}\\
 I_m^-&=x^2(x^2,y)^{m-3}+(xy^{2m-4}).\label{lat-eq:minus-m}
\end{align}
The colengths and window dimension are
\begin{equation}\label{lat-eq:window-size}
 \length(R_m/I_m^-)=m(m-1),\qquad
 \length(R_m/I_m^+)=m-1,\qquad
 \dim_\kk\Delta_{K_m}=(m-1)^2.
\end{equation}
These numbers describe the finite module window; they do not give the dimensions of its submodule varieties.
\end{theorem}
\begin{proof}
For the upper endpoint, apply \eqref{lat-eq:max-test}. The quotient functional detects the coefficient of $y^{m-2}$ modulo $x$. An element is annihilated by this functional after multiplication by every element of $R_m$ exactly when its reduction modulo $x$ is divisible by $y^{m-1}$. This proves \eqref{lat-eq:plus-m}.

For the lower endpoint, use the normal form $\sum_{j=0}^{a-1}f_j(x)y^j$ over $\kk[[x]]$. An invariant monomial in $xR_m\eta$ has $i\ge1$ and $i+2j+1=0\pmod n$. For $0\le j\le m-3$, the minimal positive $i$ is $2(m-2-j)$; these produce the generators of $x^2(x^2,y)^{m-3}$. At $j=m-2$ the minimal positive $i$ is $n$, already divisible by the first generator. For $m-1\le j\le n-1$ the minimal exponent is $2n-1-2j$; all terms except the last are divisible by $x^2y^{m-3}$. The last is $xy^{n-1}=xy^{2m-4}$. At $j=n$ the minimal exponent is $n-1$, again redundant. Adding multiples of $n$ to $i$ gives no further generators. This proves \eqref{lat-eq:minus-m}.

The upper quotient has basis $1,y,\ldots,y^{m-2}$. In the lower quotient, $x^{a-2}=0$ forces $y^a=x^{a+1}=0$. At $x$-degree zero there are $a$ monomials, and at $x$-degree one there are $a-2$. For $x$-degrees $2,3$ there are $m-3$ each, then $m-4$ each at degrees $4,5$, and so on down to one each. Their total is $a+(a-2)+2\sum_{j=1}^{m-3}j=m(m-1)$. The difference is $(m-1)^2$.
\end{proof}

The original Cartier quotient $R_m/(x)$ has representation $\kk[\mu_n]\oplus\mathbf1$. Fixing that representation selects a particular submodule Grassmannian in the window. The canonical character is essential here: it is the twisted window that has no invariant vectors.

\subsection{A uniform projective classification}
We use the one-dimensional character spaces to determine both the image of canonical coarsening and the full scheme structure of its source.

\begin{theorem}[Uniform classification of punctual length-one canonical lifts]
\label{one-thm:uniform}
Use $R_m,\Omega_m,\rho_m$ defined above. Let $\mathfrak H_m$ be the
moduli scheme of actual ideals with quotient representation $\rho_m$
whose canonical quotient is length one at the coarse cusp.
Let
\[
 \nu_r:\PP^1_{[A:B]}\longrightarrow\PP^r,
 \qquad [A:B]\longmapsto[A^r:A^{r-1}B:\cdots:B^r]
\]
be the rational normal curve parametrization, where $r=m-2$.
Then the complete canonical-coarsening morphism is
\begin{equation}\label{one-eq:global-model}
 \mathfrak H_m\simeq
 V(Bq_0)\subseteq\PP^1_{[A:B]}\times\PP^1_{[q_0:q_1]},
 \qquad c_m=\nu_r\circ\operatorname{pr}_1.
\end{equation}
In particular:
\begin{enumerate}
\item The scheme image is the rational normal curve $\nu_r(\PP^1)$;
points outside it admit no lift of the specified representation.
\item The source is a reduced nodal union of two projective lines.
One maps isomorphically onto the rational normal curve; the other is
contracted to $[1:0:\cdots:0]$.
\item The Cartier locus consists of the contracted component minus its
attachment point, and is isomorphic to $\A^1$.
\item On the chart $T_0\ne0$, the equations and ideals are
\[
 \tau_j=p^j\ (1\le j\le r),\qquad pq_0=0,
\]
\[
 \begin{array}{c|c|c}
 \text{chart}&\text{coefficient ring}&\text{universal ideal}\\\hline
 q_0\ne0&\C[\lambda],\ p=0
 &(x-\lambda y^{r+1},\ y^{2r+2})\\
 q_1\ne0&\C[p,z]/(pz)
 &(x^2-py,\ y^{r+1}-zx).
 \end{array}
\]
The transition is $z=\lambda^{-1}$, $p=0$.
\end{enumerate}
The description is functorial over nonreduced classical test bases. All its
quotients are annihilated by $(x,y)^{2r+2}$.
\end{theorem}

\subsubsection{The chart \texorpdfstring{$T_0\ne0$}{T0 nonzero}}
\begin{proof}[Proof of the affine part]
Work over a local test algebra and put $Q=R_m/I$. The canonical
condition identifies $Q_{2r}$ with a free line generated by $y^r$.
For $1\le j\le r$, multiplication by $y^{r-j}$ from $Q_{2j}$ to
$Q_{2r}$ sends $y^j$ to a basis, so it is an isomorphism of lines.
In particular $y$ generates $Q_2$ and
\begin{equation}\label{one-eq:x2py}
 x^2=p y.
\end{equation}
Evaluating the canonical generators gives $\tau_j=p^j$ for every $j$.
These are the equations of the affine moment curve.

The invariant monomial $b_0=xy^r\in\m_{B_m}$ kills $y^r$, so $xy^{2r}=0$. The element $xy^{r+1}$ has weight two. Write it as $c y$ and
multiply by $y^{r-1}$; since $y^r$ is a basis, one obtains $c=0$.
Thus
\begin{equation}\label{one-eq:nil-reduction}
 xy^{r+1}=0,\qquad
 y^{2r+2}=x^{2r+3}=p^{r+1}xy^{r+1}=0.
\end{equation}

The quotient is now spanned by
\[
 y^j\ (0\le j\le2r+1),\qquad xy^j\ (0\le j\le r).
\]
The line $Q_1$ is generated by $x,y^{r+1}$. For $0\le j<r$, the
line $Q_{1+2j}$ is generated by $xy^j,y^{r+1+j}$. The kernel line
of $y^r:Q_0\to Q_{2r}$ is generated by $xy^r,y^{2r+1}$.
Multiplication by $y$ maps each displayed generating pair to the next.
These are surjections of line bundles, hence isomorphisms.
Consequently the opens where $x$ or $y^{r+1}$ generates $Q_1$
cover the moduli problem.

On the $y^{r+1}$-open the monomials $1,y,\ldots,y^{2r+1}$ form a
basis, and $x=\lambda y^{r+1}$. Equations \eqref{one-eq:x2py} and
\eqref{one-eq:nil-reduction} give $p=0$.
On the $x$-open the basis is
\[
 1,y,\ldots,y^r,\quad x,xy,\ldots,xy^r,
\]
and $y^{r+1}=zx$. Equation \eqref{one-eq:nil-reduction} then gives
$0=xy^{r+1}=pz y$, whence $pz=0$.

Conversely, the two chart ideals give free quotients with the stated bases.
For the second chart, the two leading monomials $x^2,y^{r+1}$ are
coprime, so the monic reductions leave exactly the rectangular basis.
The plane equation reduces to
\[
 y^{2r+2}-x^{2r+3}=pz(z-p^{r+1})y=0.
\]
Every degree-$2r+2$ monomial vanishes: reduce its $x^2$ factors using
\eqref{one-eq:x2py}; the remaining high $y$ power either vanishes directly
or introduces a factor $pz$. This proves the asserted uniform finite
jet bound. The first chart has the same bound directly.
The ordinary weight counts are $2,1,\ldots,1$.

The invariant augmentation ideal in the quotient is the line spanned
by $xy^r$ on the second chart and by $y^{2r+1}$ on the first.
It annihilates $y^r$. Thus the entire coarse maximal ideal acts trivially on the canonical quotient.
The generator values are exactly $1,p,\ldots,p^r$.
The charts therefore exhaust the moduli problem, and their overlap
identifies $z=\lambda^{-1}$. Every step uses a basis of a locally free
isotypic summand or monic division, so the argument also holds over
nonreduced bases.
\end{proof}

\subsubsection{The second projective chart}
\begin{proof}[Proof of the projective completion]
It remains to show that $T_0\ne0$ and $T_r\ne0$ cover the source,
and to compute the latter chart. At a geometric point some canonical
monomial $h_j$ has nonzero image. If $0<j<r$, the weight-two elements
$x^2,y$ lie in a one-dimensional space. A nonzero product
$(x^2)^j y^{r-j}$ forces both endpoint products $y^r,x^{2r}$ to
be nonzero. For
$j=0,r$, one endpoint is already nonzero. Thus the two endpoint opens cover all geometric points and
hence the whole source, including any nilpotent structure.

On $T_r\ne0$, $x^{2r}$ generates $Q_{2r}$. Each $x^j$ for
$1\le j\le2r$ therefore generates $Q_j$. In particular $y=u x^2$ for a unique parameter $u$. The quotient is generated as an algebra by
$x$. Since $x^{n+1}$ has weight one, write $x^{n+1}=b x$.
Then all powers reduce to $1,x,\ldots,x^n$. They span a locally
free module of rank $n+1$ and hence form a basis. The invariant element
$x^n\in\m_{B_m}$ kills the canonical quotient, so $0=x^n x^{n-1}=b x^{n-1}$, forcing $b=0$. The full chart is therefore
\begin{equation}\label{one-eq:infinity-chart}
 \Spec\C[u],\qquad I=(y-u x^2,x^{n+1}),\qquad
 [T_0:\cdots:T_r]=[u^r:u^{r-1}:\cdots:1].
\end{equation}
Conversely, this ideal satisfies all the required conditions. On its overlap with
$T_0\ne0$, $u$ is invertible, $p=u^{-1}$, and $z=0$.
Thus \eqref{one-eq:infinity-chart} completes the noncontracted affine
line to $\PP^1$, with no additional exceptional fibre.
The two components and their scheme-theoretic intersection are exactly
$V(Bq_0)$, proving \eqref{one-eq:global-model}.

On the contracted component with $q_0\ne0$, the principal generator
$x-\lambda y^{r+1}$ has valuation $a$. Substitution into the plane
relation produces the formal unit $1-\lambda^{a+1}y^{n(r+1)}$, so these are Cartier ideals of colength $a$. At the attachment point
one has $(x^2,y^{r+1})$. On the noncontracted affine component the
ideals are $(x^2-py,y^{r+1})$; at its infinity they are $(y,x^a)$.
None contains an element of valuation $a$, although all have colength
$a$. Indeed all have zero linear $x$-coefficient; the least positive
value of $R_m$ is $a$, achieved only by elements with a nonzero linear
$x$-coefficient. The equality between colength and valuation for a
principal ideal follows from the finite normalization sequence and
multiplication by its generator. These ideals are therefore nonprincipal.
This proves the Cartier assertions and completes the theorem.
\end{proof}

\subsubsection{The image and the scope of the classification}
For $m=3$, $r=1$, the image is the full projective line; restricting to
$T_0\ne0$ gives $V(\tau q_0)\subset\A^1\times\PP^1$. For $m=4$ the image is a conic; in the two-parameter affine target, it is the corresponding parabola chart. For $m\ge5$ the
image is a rational normal curve inside a higher-dimensional space.
Thus most length-one canonical lattices have no lift \emph{of the
fixed representation} for $m\ge4$.

As abstract schemes, the sources in Theorem~\ref{one-thm:uniform} are all the
same nodal union of two projective lines. The degree and position of the image in the canonical-quotient space change, while the abstract moduli singularities remain the same. The source contains a non-Cartier section over
the image, and all Cartier lifts occur over one distinguished quotient.
The classification is restricted to the chosen quotient representation and length-one canonical quotients in this plane-branch family. Other representations, higher quotient lengths, and general Gorenstein curve singularities require separate analysis.

\subsection{Specialization of the lattice endpoints}\label{lat-sec:jump}
The relative Quot scheme can remain projective while the smaller lattice window fails to be flat. To see this, take $m=3$, put $A=\C[\tau]$, and use $R=\C[[x,y]]/(y^4-x^5)$ with weights $(1,2,1)$ modulo three. Define $K_\tau$ by the quotient
\begin{equation}\label{lat-eq:movingK}
\Omega\otimes A\longrightarrow A,\qquad
 y\eta\longmapsto1,\quad x^2\eta\longmapsto\tau,\quad
 \mathfrak m_B\Omega\longmapsto0.
\end{equation}
This family isolates the additional flatness and base-change hypotheses in Proposition~\ref{lat-prop:relative-window}.

Use the identification $W=R\eta$. The composite $R_A\eta\to\Omega_A\to A$ detects $\text{coefficient of }y\; +\;\tau\,\text{coefficient of }x^2$. The total upper envelope is therefore the annihilator of this functional under all $R_A$-multiplications. Since $\tau$ is regular on $A$, the coefficient of $x$ must vanish. A direct calculation gives
\begin{equation}\label{lat-eq:totalupper}
 I_A^+=(x^2-\tau y,xy,y^2).
\end{equation}
Its quotient is free on $1,x,y$ over $A$. Yet
\begin{equation}\label{lat-eq:badBC}
 I_A^+\otimes_A A/(\tau)=(x^2,xy,y^2)
 \ \subsetneq\ I_{K_0}^+=(x,y^2).
\end{equation}
Since the total quotient is flat, the tensor on the left is still an embedded ideal. Thus \eqref{lat-eq:badBC} records a failure of $\Hom_{B_A}(R_A,K_\tau)$ to commute with specialization; it is not an ambiguity in the notation for an image.

For the total lower envelope, the invariant generators give
\begin{equation}\label{lat-eq:totallower}
 I_A^-=(x^2-\tau y,xy^2).
\end{equation}
The quotient has normal forms
\[
 1,x,y,xy,y^2,y^3,\qquad \tau y^3=0,
\]
with no other $A$-linear relation. Hence
\begin{equation}\label{lat-eq:lowerjump}
                    R_A/I_A^-\simeq A^5\oplus A/(\tau).
\end{equation}
Its fibre length is five for $\tau\ne0$ and six at $0$. Over the PID $A$, the ideal itself is torsion-free and hence flat, but its quotient is nonflat. Consequently, the inclusion need not remain injective after base change.

Recomputing at a nonzero scalar $\tau$, one has
\[
 I^-_{K_\tau}=(x^2-\tau y),\qquad
 I^+_{K_\tau}=(x^2-\tau y,x^3).
\]
Their colengths are five and three. At $0$ they are six and two. Thus the dimensions of the \emph{recomputed} windows are
\[
                  \dim\Delta_{K_\tau}=2\quad(\tau\ne0),
                  \qquad\dim\Delta_{K_0}=4.
\]
The restriction of Theorem~\ref{one-thm:uniform} to this affine target is the invariant Quot model of Theorem~\ref{lat-thm:relative-quot} and remains well-defined and projective through the jump. Its finite coherent input is the nonflat module in \eqref{lat-eq:lowerjump}, tensored by the canonical character. The failure is in the relative vector-bundle window, which explains why the interval theorem over a fixed field does not extend to an unconditional flat relative Grassmann-bundle theorem.

%% file: sections/05_deformations.tex
\section{Deformations, canonical fibres, and jet detection}\label{sec:invisible}\label{sec:windows}

\subsection{An equisingular family with constant coarse curve}
We now deform the equivariant curve while keeping its simultaneous normalization marked. The parameter has residual degree $-2$, so we allow coefficient algebras with negative residual degrees.

For every $m\geq4$, set $a_0=2m-2$, $n=a_0-1$, and $B=\C[[\beta]]$. Consider
\begin{equation}\label{def-uniformcurve}
R_{m,\beta}=
B[[x,y]]/(y^{a_0}-x^{a_0+1}-\beta x^4y^{a_0-2}),
\qquad \mu_n:\ \wt(x,y)=(1,2).
\end{equation}
The coefficient $\beta$ is invariant under the original $\mu_n$ and has effective residual weight $-2$.

\begin{theorem}[Uniform coarse-invisible curve deformation]\label{def-uniform}
The family \eqref{def-uniformcurve} is flat, has Gorenstein curve fibres, and admits a finite simultaneous normalization $B[[t]]$. Its value semigroup is $\langle a_0,a_0+1\rangle$, its relative conductor is $t^{a_0n}B[[t]]$, and its normalization quotient is $B$-free of rank $a_0(a_0-1)/2$. Under $s=t^n$, its invariant ring is exactly
\begin{equation}\label{def-constantcoarse}
R_{m,\beta}^{\mu_n}=B+s^mB[[s]]\subset B[[s]].
\end{equation}
The marked equivariant curve deformation is already nontrivial to first order.
\end{theorem}
\begin{proof}
There is a unique $U(z)\in\C[[z]]$, $U(0)=1$, satisfying $U^{a_0}-zU^{a_0-2}=1$. The substitution
\begin{equation}\label{def-param}
x=t^{a_0},\qquad y=t^{a_0+1}U(\beta t^{2n})
\end{equation}
solves \eqref{def-uniformcurve}. It is equivariant for $t\mapsto\zeta t$. Both source and target are free of rank $a_0$ over $B[[t^{a_0}]]$; the map is injective by reduction at $\beta=0$ and flatness. It is finite and birational, hence $B[[t]]$ is the normalization of the reduced total family.

The elements $x^{a_0-1-j}y^j$, $0\leq j<a_0$, have leading exponents $a_0n+j$. As a $B[[t^{a_0}]]$-module, their coefficient matrix in $t^{a_0n}B[[t]]$ is the identity modulo $\beta$, so is invertible. This proves the conductor inclusion. Below $a_0n$, distinct monomials with $y$-exponent less than $a_0$ have distinct leading valuations. Triangular reduction therefore leaves a free basis indexed by the gaps of $\langle a_0,a_0+1\rangle$. There are $a_0(a_0-1)/2$ gaps. The last gap is $a_0n-1$; multiplying any nonzero element with leading exponent below $a_0n$ by an appropriate power of $t$ detects that gap. This proves that the conductor is exactly as stated, also after Artin coefficient base change.

The conductor inclusion gives $s^{a_0}B[[s]]$ inside the coarse ring after taking invariants. For $m\leq L<a_0$, the invariant monomial
\[
x^{2L-a_0-1}y^{a_0-L}
=s^LU(\beta s^2)^{a_0-L}
\]
has leading term $s^L$ and tail in exponents at least $L+2$. Descending induction from $L=a_0-1$ produces every $s^L$, $L\geq m$. Conversely, no positive invariant monomial has valuation below $nm$. Indeed, for $1\leq L<a_0$, division $nL=a_0(L-1)+(a_0-L)$ and the semigroup membership criterion give $nL\in\langle a_0,a_0+1\rangle$ exactly when $L\geq m$. This proves \eqref{def-constantcoarse}.

Finally, $x^4y^{a_0-2}$ is nonzero in the Tjurina quotient
$\C[[x,y]]/(x^{a_0},y^{a_0-1})$ because $a_0\geq6$. It has the required character, so no equivariant coordinate change kills the first-order curve deformation. The marked quotient-stack deformation is also nontrivial: finite \'etale torsors lift uniquely across nilpotent thickenings \cite{stacks}.
\end{proof}

Thus \eqref{def-uniformcurve} deforms the marked DMH curve chart while keeping its coarse ring, normalization and conductor fixed. The construction is formal-local.

\subsection{Persistence of the deepest divisor}
Specialize to $m=4$. Write $R_\beta=R_{4,\beta}$, $s=t^5$, and $C_\beta=B+s^4B[[s]]$. Define the invariant ideal
\begin{equation}\label{def-istar}
I_*=(x^6,x^5y,x^4y^2,x^2y^3,xy^4,y^5).
\end{equation}
It contains the deformed curve equation for every $\beta$, so its ordinary quotient is the constant rank-eighteen monomial quotient.

In the simultaneous normalization the six generators have valuations $33,\ldots,38$. Order them by valuation and use the $B[[t^6]]$-basis $t^{33},\ldots,t^{38}$ of $t^{33}B[[t]]$. Their coefficient matrix is the identity modulo $\beta$, hence is invertible. Consequently
\begin{equation}\label{def-istarN}
I_*R_\beta=t^{33}B[[t]].
\end{equation}
Use the canonical residue generator
\[
\eta_\beta=\frac{dx}{6y^5-4\beta x^4y^3}
=t^{-30}V(\beta t^{10})\,dt,
\]
where
\begin{equation}\label{def-V}
V(z)=\bigl(U(z)^5-\tfrac23zU(z)^3\bigr)^{-1},\qquad
U=1+\frac z6+O(z^2),\quad V=1-\frac z6+O(z^2).
\end{equation}
Both $U,V$ are invariant units.

After suppressing a common nonzero scalar in the trace identification of differentials, the fixed coarse canonical module is
\begin{equation}\label{def-omega}
\Omega=(s^{-4},s^{-3},s^{-2})C_\beta\,ds.
\end{equation}
The residue calculation identifies this module directly. The conductor contributes $B[[s]]\,ds$. The three low invariant residues are
$s^{-4}U^2V$, $s^{-3}UV$, and $s^{-2}V$. Their leading coefficient matrix on $s^{-4},s^{-3},s^{-2}$ is invertible, and the coefficient of $s^{-1}$ vanishes because $[z](UV)=0$.

\begin{proposition}[The deepest divisor in the family]\label{def-deepconstant}
Along the simultaneous curve deformation $R_\beta$, the divisor $I_*$ retains its coarse canonical lattice, and the coarse curve remains fixed:
\[
\Gamma_\omega(I_*):=(I_*\eta_\beta)^{\mu_5}
=B[[s]]\,ds=\mathfrak m_{C_\beta}\Omega.
\]
The ordinary quotient retains character ranks $(4,4,4,3,3)$ and the canonical quotient is completely semisimple of rank three.
\end{proposition}
\begin{proof}
Equation \eqref{def-istarN} and the invariant unit $V$ give
$I_*\eta_\beta=t^3B[[t]]\,dt$. Its first invariant differential is $t^4dt$, which descends to a nonzero scalar times $ds$. All successive invariant exponents give $B[[s]]ds$. The quotient of \eqref{def-omega} by that module has basis $s^{-4}ds,s^{-3}ds,s^{-2}ds$ and is annihilated by the coarse maximal ideal. Ordinary flatness follows either from the constant monomial quotient or from \eqref{def-istarN} and the free gap basis in Theorem~\ref{def-uniform}.
\end{proof}

The equivariant curve--divisor pair therefore varies with fixed canonical coarse pair. The residual degree of its coefficient is $-2$.

\subsection{Curve--divisor compensation at the Cartier stratum}
Near $(x^3)$, the same family allows an explicit compensation between curve and divisor motion. This complements the calculation at $I_*$. Work over $B=\C[[\alpha,\beta]]$ and take $I_{\alpha,\beta}=(x^3-\alpha y^4)\subset R_\beta$. The special-fibre generator is a nonzerodivisor, so the quotient is flat by the fibrewise criterion \cite{stacks}. It is finite free of rank eighteen and has the same ordinary character ranks.

For $\lambda\in B$, define the canonical lattice
\begin{equation}\label{def-Klambda}
K_\lambda=\bigl((s^{-2}-\lambda)C_\beta+sB[[s]]\bigr)ds\subset\Omega.
\end{equation}
Its quotient is free of rank three with basis $s^{-4}ds,s^{-3}ds,ds$. The only nonzero positive generator action is $s^4:s^{-4}ds\mapsto ds$, so its Jordan type is $(2,1)$.

\begin{theorem}[Canonical compensation]\label{def-compensation}
The canonical contraction is exactly
\begin{equation}\label{def-compensformula}
\Gamma_\omega(I_{\alpha,\beta})=K_{\alpha+\beta/6}.
\end{equation}
In particular, $\alpha=-\beta/6$ gives a nonconstant marked equivariant curve--divisor pair with constant coarse curve and constant canonical lattice.
\end{theorem}
\begin{proof}
In the normalization the generator is
$t^{18}(1-\alpha s^2U^4)$. After multiplication by $x\eta_\beta$, the invariant coefficient is
\[
s^{-2}V(1-\alpha s^2U^4)
=s^{-2}-\alpha-\beta/6+O(s^2).
\]
The conductor supplies $s^3B[[s]]ds$ to the contracted lattice. Multiplication by $y^3\eta_\beta$ and $x^2y^2\eta_\beta$ supplies $s$ and $s^2$ modulo that tail, hence supplies all of $sB[[s]]ds$. Removing this tail replaces the displayed generator by $s^{-2}-\alpha-\beta/6$ exactly.

Conversely, the relevant character-one part of $R_\beta$ has only the leading valuation $6$ before its next possible valuation $21$. Therefore there are no further negative or constant coefficients in the contracted lattice. This proves both inclusions in \eqref{def-compensformula}. Nontriviality of the curve component was proved by its Tjurina class; it cannot be removed by an isomorphism of the marked pairs.
\end{proof}

The coefficient $1/6$ comes from the canonical Jacobian in \eqref{def-V}. The formula shows explicitly how curve motion and divisor motion compensate under canonical coarsening. We next examine the entire fixed canonical fibre at $I_*$.

\subsection{The marked family and the comparison problems}
We continue with the curve family above, over coefficient rings of characteristic zero. Write $D=\C[[\beta]]$ and
\begin{equation}\label{win-curve}
R_\beta=D[[x,y]]/(y^6-x^7-\beta x^4y^4),\qquad
G=\mu_5,\quad \wt(x,y,\eta_\beta)=(1,2,1).
\end{equation}
The coefficient $\beta$ is $G$-invariant and has residual weight $-2$.

The parametrization from Theorem~\ref{def-uniform} is
\begin{equation}\label{win-param}
x=t^6,\qquad y=t^7U(\beta s^2),\qquad s=t^5,
\qquad U(z)^6-zU(z)^4=1,\quad U(0)=1.
\end{equation}
Its coarse subring is
\begin{equation}\label{win-coarse}
C_D=R_\beta^G=D+s^4D[[s]]\subset D[[s]],
\end{equation}
and its conductor is $t^{30}D[[t]]$. The natural canonical residue is
\begin{equation}\label{win-eta}
\eta_\beta=t^{-30}V(\beta s^2)\dd t,\qquad
V(z)=\left(U(z)^5-\frac23zU(z)^3\right)^{-1}.
\end{equation}
A common nonzero scalar in passing from $\dd t$ to $\dd s$ is suppressed. The fixed coarse canonical module is
\[
\Omega_D=(s^{-4},s^{-3},s^{-2})C_D\dd s
=D\langle s^{-4},s^{-3},s^{-2}\rangle\dd s+D[[s]]\dd s.
\]
The three low invariant residues $s^{-4}U^2V$, $s^{-3}UV$, $s^{-2}V$ have an invertible triangular coefficient matrix, and the coefficient of $s^{-1}$ vanishes since $[z](UV)=0$. The conductor supplies the tail.

Fix ordinary quotient character ranks $(4,4,4,3,3)$, of total length eighteen. The completely semisimple canonical lattice is $K_{\rm ss}=D[[s]]\dd s=\mathfrak m_{C_D}\Omega_D$. The deepest monomial ideal
\[
I_*=(x^6,x^5y,x^4y^2,x^2y^3,xy^4,y^5)=t^{33}D[[t]]
\]
has this contraction. These equalities identify the embedded lattices and subrings in the marked normalization.

The following three comparison problems will play distinct roles:
\begin{enumerate}[nosep]
\item the actual fixed-lattice fibre $\mathscr F_\beta$, with $(I\eta_\beta)^G=K_{\rm ss}$;
\item the larger support-cycle fibre, whose canonical quotient need only have cycle $3[o]$;
\item the unrestricted source and its finite normalization comparison.
\end{enumerate}
Theorem~\ref{win-fibre} treats (1) to all orders. Section~\ref{sec:second-order} treats the pair (3),(2) through second order, modulo $\beta^3$. Throughout, isomorphisms involving the coarse ring preserve its identification with \eqref{win-coarse}. Over a nonreduced Artin base, $D'[[t]]$ denotes the specified simultaneous comparison obtained by base change; absolute normalization is not intended.

\subsection{The finite lattice window and an all-order connection}
The fixed-fibre comparison reduces to a representation on a constant sixteen-dimensional quotient. We first compute its two endpoints and then construct a gauge identifying the multiplication operators.
\subsubsection{The induced and coinduced endpoints}
Dividing the induced and coinduced canonical lattices by the invertible generator $\eta_\beta$ gives ideals $I^-_\beta,I^+_\beta$ in the fixed normalization.
\begin{proposition}\label{win-endpoints}
For every local Artin $D$-algebra $D'$, the endpoints are
\begin{align}
I^-_\beta&=t^{34}D'[[t^5,t^6,t^7]],\label{win-low}\\
I^+_\beta&=\bigoplus_{j=0}^4t^{q_j}D'[[s]],
& (q_0,q_1,q_2,q_3,q_4)&=(20,21,27,28,34).\label{win-up}
\end{align}
They are constant as embedded lattices in $D'[[t]]$. Their quotient is free of rank sixteen, with normalized monomial basis
\begin{equation}\label{win-window}
\mathcal W=\langle e_v=t^v:\ v\in\mathcal V\rangle,\quad
\mathcal V=\{20,21,25,26,27,28,30,31,32,33,35,36,37,38,42,43\}.
\end{equation}
\end{proposition}
\begin{proof}
Since $K_{\rm ss}/\eta_\beta=t^{34}V^{-1}D'[[s]]$, its induced $R_\beta$-span is \eqref{win-low}: adjoining $s$ to $R_\beta$ gives $D'[[t^5,t^6,t^7]]$, because $U,V$ are units in $D'[[s]]$.

To compute the upper endpoint, require a weight-$j$ element $z$ to make every invariant product $rz\eta_\beta$ regular in $s$. The least valuations in the five character parts of $R_\beta$ are $(0,6,7,13,14)$. Since a differential $t^h\dd t$ is invariant and regular on the coarse normalization precisely when $h\equiv4\pmod5$ and $h\ge4$, one obtains the threshold $v(z)+v(r)\ge34$. Applying it in each opposing character gives \eqref{win-up}. Leading coefficients are units, so triangular elimination proves the same assertion over Artin rings, including elements with nilpotent coefficients. The threshold lattices lie in $R_\beta$: the conductor supplies all exponents at least thirty; the finitely many exponents below thirty are supplied by monomials and their tails can be removed using that conductor.

The semigroup $\langle5,6,7\rangle$ contains every integer at least ten. Subtracting the exponents in \eqref{win-low} from \eqref{win-up} gives exactly \eqref{win-window}. This computes both endpoints directly from multiplication and residue coefficients, without an appeal to relative $\Hom$ base change.
\end{proof}

\begin{remark}[Dimension and incidence presentation]\label{win-erratum}
The central lower ideal may be written
$ I^-_0=(xy^4,x^3y^3,x^5y^2,y^7,x^9)$.
The relation $x^8=y^2(xy^4)$ shows that $x^8$ is already zero in its quotient; in the deformed ring the corresponding identity is $x^8=(y^2-\beta x^4)(xy^4)\in I^-_\beta$. Thus the endpoint colengths are $25$ and $9$, the window dimension is $16$, and the submodule character ranks are $(1,1,2,3,0)$. Appendix~\ref{win-incidence} gives an equivalent incidence presentation of this quotient.
\end{remark}

\subsubsection{A polynomial gauge for the window representation}
On the basis \eqref{win-window}, let $S_j$ send $e_v$ to $e_{v+j}$ when that exponent is in $\mathcal V$, and to zero otherwise. Then
\[
X_\beta=X=S_6,\qquad Y_\beta=Y+\frac{\beta}{6}Z,\qquad Y=S_7,\quad Z=S_{17}.
\]
The formula is exact because $U(z)=1+z/6+O(z^2)$ and all shifts by $27$ or more vanish in this window. Moreover $(x,y)^4\mathcal W=0$, since its least possible shift is $24$ and the window ends at valuation $43$, with lower-endpoint tail starting at $44$. Thus the full power-series action factors through a common finite quotient; checking $X_\beta$ and $Y_\beta$ checks the entire module action.

Define $T\in\operatorname{End}(\mathcal W)$ by its five nonzero values:
\begin{equation}\label{win-T}
\begin{aligned}
Te_{20}&=-\tfrac16e_{30},&Te_{25}&=-\tfrac13e_{35},&Te_{26}&=-\tfrac16e_{36},\\
Te_{28}&=\tfrac16e_{38},&Te_{32}&=-\tfrac16e_{42}.
\end{aligned}
\end{equation}
Direct multiplication gives
\begin{equation}\label{win-connection}
T^2=ZT=TZ=0,\qquad [X,T]=0,\qquad [Y,T]=-Z/6.
\end{equation}
Consequently
\begin{equation}\label{win-gauge}
P_\beta=1+\beta T,\quad P_\beta^{-1}=1-\beta T,
\quad X_\beta P_\beta=P_\beta X,
\quad Y_\beta P_\beta=P_\beta Y.
\end{equation}
All entries of $T$ shift valuation by ten, so this is $\mu_5$-equivariant.

\begin{theorem}[Constancy of the fixed canonical fibre]\label{win-fibre}
Let $\mathscr F_\beta$ parameterize actual invariant ideals with flat ordinary quotient of character ranks $(4,4,4,3,3)$ and exact contraction $(I\eta_\beta)^G=K_{\rm ss}$. Then
\[
\mathscr F_\beta\simeq\mathscr F_0\times\Spf\C[[\beta]]
\]
as scheme-valued functors on local Artin coefficient rings. The isomorphism fixes the section $I_*$ and is valid on arbitrary local Artin test bases, including nonreduced ones.
\end{theorem}
\begin{proof}
Reynolds averaging, or directly the invariant residue condition, gives the exact interval
\[
(I\eta_\beta)^G=K_{\rm ss}\quad\Longleftrightarrow\quad
I^-_\beta\subseteq I\subseteq I^+_\beta.
\]
The invariant parts of the two endpoints agree, so this equivalence also holds over nonreduced test bases. Thus the functor is the closed submodule functor of $\mathcal W$, with subbundle ranks $(1,1,2,3,0)$ and stability under $X_\beta,Y_\beta$. The ordinary flatness conditions are equivalent to these subbundle conditions: both outer endpoint quotients are finite free, so the relevant short exact sequences preserve flatness. Equation~\eqref{win-gauge} transports these subbundles in both directions and is compatible with every coefficient base change. For $I_*$ the submodule is spanned by exponents in $\mathcal V$ at least thirty-three. The operator $T$ kills all seven such vectors.
\end{proof}

\begin{remark}[A finite-window connection principle]
For any finite free representation $X,Y+\beta Z$, an operator $T$ satisfying $[T,X]=0$, $[T,Y]=Z$, and $[T,Z]=0$ gives a formal gauge $\exp(\beta T)$. The gauge trivializes the associated stable-submodule functors; the underlying algebra and a larger Hilbert scheme may still vary. In our case the scaled $Z/6$ and $T^2=0$ reduce the exponential to the polynomial \eqref{win-gauge}. The connection $\partial_\beta-T$ is horizontal for this representation.
\end{remark}

\subsection{A length-four jet invariant for rank-one lattices}
Theorem~\ref{win-fibre} describes constancy within the fixed canonical fibre. A different finite window recovers the curve parameter from one of its eigensheaves.

Let $D'$ be any commutative $\C$-algebra. A plane in the open big cell of $\Gr(2,D'[s]/s^4)$ has a unique echelon basis
\begin{equation}\label{win-plane}
F=1+A s^2+B s^3,\qquad G=s+C s^2+D_1s^3.
\end{equation}
Here a length-four jet means reduction modulo $s^4$; this is a $3$-jet in the convention that uses $s^{n+1}$ for $n$-jets. The unipotent group of units $1+h_1s+h_2s^2+h_3s^3$ acts by multiplication.
\begin{theorem}[Unit orbits of length-four jets]\label{win-jet}
The invariant
\begin{equation}\label{win-j}
\mathfrak j=A-D_1+C^2
\end{equation}
classifies the unit orbits functorially over $D'$. Each orbit has the unique representative $\langle1+\mathfrak j s^2,s\rangle$. Thus the big cell is equivariantly a product of the unipotent unit group and an affine line, with a specified global slice. Allowing scalar units changes stabilizers but not this orbit invariant.
\end{theorem}
\begin{proof}
After multiplying by $h$ and restoring echelon form,
\[
C'=C+h_1,\quad D_1'=D_1+h_1C+h_2,\quad
A'=A+h_2-h_1C-h_1^2.
\]
These formulas preserve \eqref{win-j}. The unique unit carrying \eqref{win-plane} to the slice is
\begin{equation}\label{win-slice}
1-Cs+(C^2-D_1)s^2+(CA-B)s^3.
\end{equation}
Direct multiplication proves existence and uniqueness over an arbitrary coefficient ring, using only polynomial operations in the coefficients.
\end{proof}

Throughout this comparison, the normalization parameter $s$ and the coarse-ring embedding remain fixed; nonlinear reparametrizations of $s$ are not allowed.

\subsubsection{Detecting the curve parameter for every \texorpdfstring{$m$}{m}}
Let $m\ge4$, $a_0=2m-2$, $n=a_0-1$, and consider the simultaneously normalized family of Theorem~\ref{def-uniform}
\[
R_{m,\beta}=D[[x,y]]/(y^{a_0}-x^{a_0+1}-\beta x^4y^{a_0-2}),\quad
x=t^{a_0},\quad y=t^{a_0+1}U(\beta s^2),\quad s=t^n,
\]
where $U^{a_0}-zU^{a_0-2}=1$. Its fixed coarse ring is $D+s^mD[[s]]$.
\begin{proposition}\label{win-eigen}
The character-two eigensheaf has the exact fractional-lattice description
\begin{equation}\label{win-eigplane}
(R_{m,\beta})_{\chi^2}
=t^{a_0+1}\mathcal P_{m,\beta},\qquad
\mathcal P_{m,\beta}=D\,U(\beta s^2)+Ds+s^mD[[s]].
\end{equation}
Its four-jet invariant is
\begin{equation}\label{win-betaj}
\mathfrak j(\mathcal P_{m,\beta})=\frac{\beta}{2m-2}.
\end{equation}
Over any local Artin coefficient ring, two members with parameters $\beta_1,\beta_2$ are isomorphic as modules over the \emph{fixed marked coarse ring} only if $\beta_1=\beta_2$.
\end{proposition}
\begin{proof}
Among the exponents $a_0+1+nL$ in $\langle a_0,a_0+1\rangle$, those with $0\le L<m$ are exactly $L=0,1$. They are provided by $y$ and $x^2$, respectively. For $m\le L<a_0-1$, the monomials with $j=a_0+1-L$ and $i=a_0+1-2j$ give consecutive such exponents. The conductor supplies all $L\ge a_0-1$. Descending induction from the conductor removes their unit tails and proves \eqref{win-eigplane}. This argument is valid over Artin coefficients because the leading coefficients are units.

Since $U(z)=1+z/a_0+O(z^2)$, the four-jet plane is $\langle1+(\beta/a_0)s^2,s\rangle$, proving \eqref{win-betaj}. Any coarse-linear isomorphism between two full fractional lattices extends after inverting $s^m$ to multiplication by an element of $D'((s))^\times$. Both lattices generate $D'[[s]]$ over the normalization, so this multiplier preserves $D'[[s]]$ and is a unit. Theorem~\ref{win-jet} therefore implies equality of the parameters. The converse is immediate.
\end{proof}
This is the smallest truncation that separates the family: modulo $s^3$, multiplication by $1+(\la'-\la)s^2$ identifies $\langle1+\la s^2,s\rangle$ with $\langle1+\la's^2,s\rangle$. The length-four jet is therefore the first to distinguish the parameters.

\subsection{Ordinary and canonical duals}
Return to $m=4$ and put $\la=\beta/6$, $C=D'+s^4N$, $N=D'[[s]]$. Then $P_\la=D'(1+\la s^2)+D's+s^4N$. Let $\Omega'=(1,s,s^2)C$, a fractional canonical module obtained from $\Omega$ by a common shift and by omitting $\dd s$.
\begin{proposition}\label{win-duals}
The ordinary duals are constant:
\[
(C:P_\la)=s^4N,\qquad P_\la^{\vee\vee}=N.
\]
The canonical dual retains the parameter:
\[
(\Omega':P_\la)=D'+D'(s-\la s^3)+s^4N,
\qquad
(\Omega':(\Omega':P_\la))=P_\la.
\]
Multiplication by $1+\la s^2$ identifies the displayed canonical dual with $P_\la$ itself.
\end{proposition}
\begin{proof}
The presence of a unit in $P_\la$ forces any multiplier in either dual to lie in $N$. Write its first four coefficients as $h_0+h_1s+h_2s^2+h_3s^3$. Requiring $hs,h(1+\la s^2)\in C$ forces all four coefficients to vanish. For $\Omega'$, whose only missing coefficient below four is that of $s^3$, the conditions are instead $h_2=0$, $h_3=-\la h_1$. The second canonical dual follows by the same calculation. Finally $(C:s^4N)=N$, and multiplication by $1+\la s^2$ gives the stated lattice identity. These are exact colon computations, including nilpotent coefficients.
\end{proof}
The calculation exhibits the distinction between $\OO$-duality and canonical duality: the former erases this parameter, whereas the latter retains it. The finite canonical duality used here is standard \cite{stacks}.

\subsubsection{The ordinary pushforward of the compensated Cartier family}
Return to the invariant Cartier ideals
\begin{equation}\label{win-cartier}
I_\beta=(x^3+\tfrac\beta6 y^4)\subset R_\beta.
\end{equation}
Their quotients are flat of rank eighteen, by the relative nonzerodivisor criterion \cite{stacks}. Theorem~\ref{def-compensation} gives $\Gamma_\omega(I_\beta)=K_0$. The cancellation can also be seen directly in \eqref{win-eta}: the negative leading canonical coefficient is
\[
s^{-2}V(\beta s^2)(1+\la s^2U(\beta s^2)^4)=s^{-2}+O(s^2),
\]
and the full lattice contains $sN\dd s$, so removing the tail gives the exact equality.

The generator of \eqref{win-cartier} has character three. Its ordinary invariant pushforward therefore uses character two in $R_\beta$; canonical twisting uses character one. Consequently
\begin{equation}\label{win-ordinaryJ}
\begin{split}
J_\beta=(I_\beta)^G
&=D'\left(s^5+\frac\beta3s^7\right)
 +D'\left(s^6+\frac\beta6s^8\right)+s^9N\subset C,\\
J_\beta^{\vee\vee}&=s^5N.
\end{split}
\end{equation}
Indeed, $I_\beta=t^{18}(1+\la s^2U^4)R_\beta$ and \eqref{win-eigplane} imply $J_\beta=s^5(1+\la s^2U^4)P_\la$. Modulo the fixed $s^9N$ tail this is exactly \eqref{win-ordinaryJ}. Its quotient in $C$ is free of rank four. The normalized plane has $(A,B,C,D_1)=(2\la,0,0,\la)$, so its jet invariant is $\la$.

For this Cartier family, the following data remain constant:
\[
\text{coarse curve},\quad\text{canonical lattice},\quad
\text{ordinary reflexive hull},\quad\text{normalization and conductor},
\]
while the ordinary coarse ideal varies and its isomorphism class as a module over the marked coarse ring detects $\beta$.

%% file: sections/06_conductor.tex
\section{A rank-three chart and its conductor correction}\label{sec:chart}\label{sec:conductor}

\subsection{The unrestricted chart at the critical point}\label{chart-sec:ambient}
We first construct a monomial-basis chart on the algebraic curve $\Spec\C[x,y]/(y^6-x^7)$, then complete it at $I_*$ to recover the formal-local problem. On this algebraic chart both the canonical lattice and its support may vary, and quotients supported away from the origin are allowed. We compute the normalization and conductor on this algebraization before completing at $[I_*]$. Recall the ideal $I_*$ from \eqref{def-istar}; its quotient has basis
\begin{equation}\label{eq:chartbasis}
\{1,x,x^2,x^3,x^4,x^5;\quad y,xy,x^2y,x^3y,x^4y;\quad
 y^2,xy^2,x^2y^2,x^3y^2;\quad y^3,xy^3;\quad y^4\}.
\end{equation}
The character ranks are $(4,4,4,3,3)$. This is the eighteen-dimensional ordinary quotient, distinct from the sixteen-dimensional canonical-lattice window of Section~\ref{sec:windows}.

The seven border monomials admit twenty-four homogeneous tail coefficients. Seventeen eliminations with constant coefficients reduce the commutator equations to seven free planar coordinates, which we denote by $a,b,c,d,e,f,g$. These chart coordinates are unrelated to the earlier curve parameter $a_0=2m-2$.

Set
\begin{equation}
\begin{aligned}
H&=be-bg+cd-cf,\\
T&=bd-c+e,\\
U&=bH+(f-d)(be+cd)-c(e-g),\\
E_3&=a(b+d-f)-df-g,\\
E_4&=a(c+e-g-bf)-H-ef.
\end{aligned}
\label{chart-eq:ambientcoeff}
\end{equation}
\begin{theorem}[Equations of the fixed-curve chart]\label{chart-thm:ambient}
The unrestricted monomial-basis chart on the algebraic curve at $I_*$ is
\begin{equation}
\Spec\frac{\kfield[a,b,c,d,e,f,g]}
{(a(a-f)U,\ a(a-f)TH,\ E_3,\ E_4)}.
\label{chart-eq:ambient}
\end{equation}
Its completion at the origin represents the full formal deformation problem for invariant ideals at $I_*$. Over the coefficient ring, the quotient is free on the eighteen displayed order-ideal monomials.
\end{theorem}
\begin{proof}
The seven border reductions are
\begin{align*}
y^5&=ax^5+a(b-f)x^3y-a(bf-c+g)xy^2,\\
xy^4&=abx^4-a(bf-c)x^2y+aHy^2,\\
x^2y^3&=-a(b^2-c)x^3+a(b^2f-bc+be+cd-cf)xy+by^4,\\
x^3y^3&=acx^4+a(be+cd-cf)x^2y+abHy^2,\\
x^4y^2&=-a(bc-be-cd)x^3+a(b^2e-b^2g+bcd-c^2+cg)xy+cy^4,\\
x^5y&=aUx^2-aTHy+dx^3y^2+exy^3,\\
x^6&=fx^4y+gx^2y^2+Hy^3.
\end{align*}
To obtain these reductions, the verifier starts with all homogeneous tails and eliminates the full commutator system using seventeen invertible constant linear coefficients. The remaining seven parameters are free: no further planar equations survive. Conversely, the resulting multiplication matrices commute identically, and the order-ideal monomials evaluated on $1$ give the corresponding independent basis vectors. This establishes the planar quotient chart over arbitrary coefficient rings, including nonreduced ones.

The vector $(Y^6-X^7)1$ has only four nonzero coefficients: $a(a-f)U$ on $x^2$, $-a(a-f)TH$ on $y$, $E_3$ on $x^3y^2$, and $E_4$ on $xy^3$. Since $X,Y$ commute and $1$ is cyclic, setting these four coefficients to zero imposes the entire curve relation. Hence \eqref{chart-eq:ambient} is the defining ideal itself, without saturation or passage to the radical. The displayed monomials remain a basis in every local Artin deformation by Nakayama's lemma. In a deformation of $I_*$, multiplication by $x$ and $y$ is nilpotent on the special fibre; nilpotence of the maximal ideal of the Artin base then makes both operators nilpotent on the whole quotient. The polynomial-ring action therefore extends uniquely to the power-series ring. This identifies the completed chart with the full formal deformation problem.
\end{proof}

\subsubsection{The fixed fibre and the tangent map}
The fixed semisimple fibre is obtained by setting all four coarse multiplication matrices on the canonical basis $(y^2,x^2y,x^4)\eta$ equal to zero. After eliminating $g$ with $E_3$, its ideal is
\begin{equation}
(ab,ac,S),\quad
S=a^2d-a^2f-adf-ae+bdf+be+cd-cf+ef.
\label{chart-eq:fixedfromambient}
\end{equation}
An invertible change of variables makes the local structure more transparent. Put $f=t+a$, $e=E+ad+a^2$; modulo $(ab,ac)$, $S$ becomes $(t+b)E+(d-t)c+bdt-a^3$. The verifier checks the preceding ideal equality in the original coordinates, before this change and without taking radicals.

The same matrices also determine the tangent map. Every entry has degree at least two at the origin, so the differential of canonical coarsening at $I_*$ vanishes. The source tangent space has dimension six, since only $E_3$ has a nonzero linear part. The target tangent space has dimension twelve:
\[
T_{K_{\mathrm{ss}}}\operatorname{Quot}_\Omega
=\operatorname{Hom}_B(\mathfrak m_B\Omega,\Omega/\mathfrak m_B\Omega),
\]
and, in the differential embedding fixed in Section~\ref{sec:windows}, $\mathfrak m_B\Omega=\kfield[[s]]\,\dd s$ has the four minimal $B$-generators $\dd s,s\,\dd s,s^2\,\dd s,s^3\,\dd s$. The fixed-fibre tangent space therefore also has dimension six. This first-order vanishing leaves room for nonconstant coarsening in higher order.

\subsection{Support equations and conductor reconstruction}
Recall the curve and its action:
\[
 R=\C[[x,y]]/(y^6-x^7),\qquad G=\mu_5,\qquad\operatorname{wt}(x,y,\eta)=(1,2,1).
\]
We consider invariant ideal inclusions whose ordinary quotients have character ranks $(4,4,4,3,3)$. The associated canonical invariant quotient has rank three. Our distinguished point is
\[
 I_*=(x^6,x^5y,x^4y^2,x^2y^3,xy^4,y^5).
\]
We work on the algebraic monomial-basis chart of Theorem~\ref{chart-thm:ambient}, whose completion at $[I_*]$ represents the formal-local problem. To simplify its equations, eliminate the variable $g$ using $E_3$.

Thus, in $S=\C[a,b,c,d,e,f]$, set
\begin{align*}
 g&=a(b+d-f)-df,&H&=be-bg+cd-cf,\\
 T&=bd-c+e,&U&=bH+(f-d)(be+cd)-c(e-g),\\
 E&=a(c+e-g-bf)-H-ef,&I&=(E,a(a-f)U,a(a-f)TH).
\end{align*}
Let $A=S/I$ and write
\[
 C=\begin{pmatrix}0&0&H\\1&0&g\\0&1&f\end{pmatrix},\qquad
 M=a\bigl(HI_3+(c-bf)C+bC^2\bigr).
\]
The canonical punctual-support ideal and its preimage in the polynomial ring are
\[
 \frs=(\tr M,\tr MC,\tr MC^2)\subset A,\qquad J=I+(\tr M,\tr MC,\tr MC^2)\subset S.
\]
The support fibre therefore has coordinate ring $A_0=S/J$. To see that these equations suffice over arbitrary bases, recall that the four invariant monomials $xy^2,x^3y,x^5,y^5$ act on the canonical quotient by $M,MC,MC^2,MC^3$. Reduction modulo $I$ gives $M(C^4-M)=0$, while $C^3=fC^2+gC+HI_3$. Consequently every positive-degree invariant operator, and every positive power of one, lies in the $A$-span of $M,MC,MC^2$. Vanishing of the three displayed traces therefore forces all such power traces to vanish. The rank-three Newton identities in characteristic zero give characteristic polynomial $T^3$ for every element of the coarse maximal ideal, and the converse is immediate. These are thus the full determinant-support equations, including over nonreduced bases. The action matrices and their identities are reconstructed in \anc{anc/source_chart/}.

Let $N=\overline A$ and $\frc=(A:N)$. Restricting the conductor square to the support fibre produces the correction
\begin{equation}\label{deep-eq:actual}
 \cK=\frac{\frc\cap\frs N}{\frc\cap\frs}\ \subset A_0.
\end{equation}
This correction need not equal $\ker(A_0\to N/\frs N)$; the conductor condition is an additional requirement. Throughout this construction we normalize the unrestricted reduced source. The support fibre, which may be nonreduced, is imposed afterwards.

\begin{lemma}[Conductor reconstruction after restriction]\label{lem:conductor-square}
For any reduced ring $A$ with finite normalization $N$, conductor $\mathfrak c$, and ideal $\mathfrak s$, let
\[
P=(N/\mathfrak sN)\times_{N/(\mathfrak c+\mathfrak sN)}A/(\mathfrak c+\mathfrak s).
\]
The natural map $A/\mathfrak s\to P$ is surjective, with square-zero kernel given by \eqref{deep-eq:actual}.
\end{lemma}
\begin{proof}
Represent a compatible pair by $n\in N$ and $a\in A$. Compatibility means $n-a=c+z$ for $c\in\mathfrak c$ and $z\in\mathfrak sN$. The element $a+c\in A$ represents this pair, proving surjectivity. The kernel is
\[
\frac{(\mathfrak c+\mathfrak s)\cap\mathfrak sN}{\mathfrak s}
=\frac{\mathfrak c\cap\mathfrak sN}{\mathfrak c\cap\mathfrak s}.
\]
If $c_1,c_2\in\mathfrak c\cap\mathfrak sN$, write $c_2=\sum s_jn_j$. Since $c_1n_j\in A$, their product belongs to $\mathfrak s$, so the kernel is square zero. This calculation is the algebraic form of the conductor square in \cite{ferrand}.
\end{proof}

The correction is supported on a union of two surfaces. Put $\rho=ad-af-df$ and define
\begin{align*}
 \Sigma&=V(b,c,e,\rho),&\Pi&=V(a,c,e,f),\\
 L_a&=V(b,c,d,e,f),&L_b&=V(a,c,d,e,f),\\
 T_0&=S/(c,e,ab,bf,\rho),&\cT&=\Spec T_0=\Sigma\cup\Pi.
\end{align*}
The two surfaces meet along the $d$-axis. The two lines meet at the origin and lie respectively in $\Sigma$ and $\Pi$.

\begin{theorem}[The correction module and its dimension filtration]\label{deep-thm:main}
For the algebraic chart $A=S/I$ above, there is an isomorphism of ungraded $A_0$-modules
\begin{equation}\label{deep-eq:main}
 \cK\simeq\omega_{T_0}\oplus\C[a]\oplus\C[b].
\end{equation}
The ring $T_0$ is reduced and Cohen--Macaulay of dimension two, with
\[
 \omega_{T_0}\simeq(d,f)\subset T_0.
\]
The canonical dimension filtration of the correction is
\[
 \cK_{\le0}=0,\qquad
 \cK_{\le1}\simeq\C[a]\oplus\C[b],\qquad
 \cK/\cK_{\le1}\simeq\omega_{T_0}.
\]
The annihilator of the correction in $S$ is $(c,e,ab,bf,\rho)$. Consequently $\cK$ is sequentially Cohen--Macaulay, is not pure, and has exactly four associated primes: those of $\Sigma,\Pi,L_a,L_b$.

For weights $\deg(a,b,d,f)=1$, $\deg(c,e)=2$, and convention $M(s)_j=M_{j+s}$, the precise graded isomorphism is
\begin{equation}\label{deep-eq:graded}
 \cK\simeq\bigl(\omega_{T_0}\oplus\omega_{L_a}\oplus\omega_{L_b}\bigr)(-7).
\end{equation}
The displayed module splitting depends on a choice. The square-zero algebra extension $0\to\cK\to A_0\to A_0/\cK\to0$ is nonsplit near the generic point of $L_a$. All conclusions persist after completion at $[I_*]$, with every ring and module replaced by its completion. In particular, the line modules become $\C[[a]]$ and $\C[[b]]$.
\end{theorem}

To pass to the completion, we use excellence of finite-type $\C$-algebras: the finite normalization remains normal after completion and normalizes the completed reduced source \cite[Tags 07QW and 0BFK]{stacks}. Flatness of completion preserves the finite module kernels and intersections below. For the conductor, the identity $\frc=\Ann_A(N/A)$ reduces the claim to flat base change for Hom from the finitely presented module $N/A$. It therefore suffices to prove the algebraic identities and then complete them.

The proof proceeds from the normalization to the conductor, and only then identifies the correction module and its filtration. The polynomial identities are computer-assisted steps in this argument; the ancillary files supply their finite computations, reconstruction formulas, variable orders and recorded outputs.

\subsection{Normalizing the seventh component in two steps}
The component decomposition uses the polynomials
\[
 Y=a(b+d-f)+bd-c-d^2,\quad K=c^2-bfc-b^2g,\quad
 W=(e-g)^2-f(f-d)(e-g)-g(f-d)^2
\]
and the ideals
\begin{align*}
 P_0&=(a,E),&P_1&=(a-f,E),\\
 P_2&=(T,c-b^2,E),&P_3&=(T,Y,E),\\
 P_4&=(b,c,E),&P_5&=(d-f,e-g,E),\\
 P_6&=(E,H,U,K,W).
\end{align*}
\begin{proposition}[Components of the unrestricted source]\label{deep-prop:components}
There is an equality of ideals $I=\bigcap_{i=0}^6P_i$. Every $P_i$ is prime, and the first six quotient rings are normal. The dimensions are four for $i=0,1$ and three for $i=2,3,4,5,6$.
\end{proposition}
\begin{proof}
Appendix~\ref{app:verification} proves the ideal intersection using complete-intersection unmixedness, explicit localization identities and generic reducedness. The finite identities and Jacobian witnesses are checked over $\mathbb Q$ by \anc{anc/components/verify_components.py}; the proof does not infer the decomposition from a sample of points. The first six factors have the normal forms listed in Appendix~\ref{app:verification}. In particular, $P_0,P_1$ are hypersurfaces with singular locus of codimension at least two, $P_2,P_3$ have the form $uvw\pm z^3=0$, and $P_4,P_5$ are products of an $A_2$ surface with an affine line. Each is Cohen--Macaulay and regular in codimension one, so Serre's criterion gives normality. For $P_6$, the finite birational presentation below proves primality independently of the subsequent normality argument. Together these results establish the proposition.
\end{proof}
Put $B=S/P_6$.

\subsubsection{The first integral root}
In $\operatorname{Frac}(B)$ let $v=c/b$. Then
\[
 c=bv,\qquad e=v^2-dv,\qquad g=v^2-fv.
\]
The intermediate ring is exactly
\begin{equation}\label{deep-eq:B1}
 B_1=B[v]=\frac{\C[a,b,d,f,v]}{(F_1,F_2)},
\quad
\begin{cases}
F_1=ab+ad-af-df+fv-v^2,\\
F_2=ab(2v-f)-v^3.
\end{cases}
\end{equation}
This extension is finite because $v^2-fv-g=0$.

The two equations form a regular sequence. On $b(d-f)\ne0$ they recover the dense root presentation of $P_6$, and eliminating $v$ gives the five generators above. The irreducibility certificate in \anc{anc/components/} verifies the domain assertion on this open set. To rule out additional components, consider the omitted loci. On $b=0$ one has $v=0$ and $a(d-f)-df=0$, a set of dimension two. On $d=f$, eliminating $ab$ gives $(v-f)^3=0$, hence set-theoretically $v=f$, $ab=f^2$, again a set of dimension two. Unmixedness and generic reducedness therefore give a three-dimensional domain. Eliminating $v$ together with $c-bv,e-v^2+dv$ recovers $P_6$.

\subsubsection{The endomorphism ring and normality}
Let $\mathfrak j=(a,f,v)\subset B_1$. Since $B_1/\mathfrak j=\C[b,d]$ is Cohen--Macaulay of dimension two, the depth lemma shows that $\mathfrak j$ is MCM. The ideal quotient is
\begin{equation}\label{deep-eq:colon}
 (v\mathfrak j:_{B_1}\mathfrak j)=(v,ab(b+d)).
\end{equation}
Hence the standard endomorphism-ring formula \cite{gls} gives
\begin{equation}\label{deep-eq:normalizer}
 B_2=\End_{B_1}(\mathfrak j)=B_1+B_1h,
 \qquad h=\frac{ab(b+d)}v.
\end{equation}
The fraction denotes an element of the common fraction field; $v$ is nonzero in the domain $B_1$. Put $C_0=ab-ad+af+2bd-bf+df$. Multiplication is governed by
\begin{equation}\label{deep-eq:multiplication}
 v^2=fv+g,\quad vh=ab(b+d),\quad
 h^2=C_0h+(b-f)ab(b+d)-bd(b+d)v.
\end{equation}
Thus $B_2$ is generated by $1,v,h$ as a $B$-module. A complete presentation is the saturated graph
\[
 \C[a,b,d,f,v,h]\Big/\bigl((F_1,F_2,vh-ab(b+d)):v^\infty\bigr).
\]
The saturation is essential: the three-equation ideal alone does not give this presentation. The electronic certificate records all nine defining polynomials of the graph.

\begin{proposition}\label{deep-prop:normal}
The ring $B_2$ is normal. Consequently $B_2=\overline B$ and
\[
 \overline A=\prod_{i=0}^5S/P_i\ \times B_2.
\]
\end{proposition}
\begin{proof}
The construction makes $B_2$ finite and birational. To prove $S_2$, use a finite presentation of $\mathfrak j$ to express $\Hom(\mathfrak j,\mathfrak j)$ as a kernel between finite sums of MCM modules. At local points of dimension at least two, the image is torsion-free and has depth at least one, so the kernel has depth at least two. Localizing gives $S_2$ for the module, and this condition transfers to the finite algebra.

The radical of the Jacobian singular ideal of $B_1$ is the intersection of
\[
 (a,f,v),\quad(b,d,f,v),\quad(a,b,d,v),\quad
 (b-a,d-a,f-a,v-a).
\]
The certificate checks both inclusions, with powers at most three witnessing the reverse inclusion. Of these four primes, only $\mathfrak j$ has codimension one. At its generic point, work over $\C(b,d)$. The first equation eliminates $a$, and the quadratic initial form of the second is a nonzero scalar multiple of $f(2v-f)$. This is an ordinary split node, whose maximal-ideal endomorphism ring is its normalization and has two regular branches. Equivalently, modulo $\mathfrak j$ the element $h$ satisfies $h(h-2bd)=0$ and separates the branches. At every other codimension-one point, $B_1$ is already regular and hence integrally closed. Thus $B_2$ satisfies $R_1$, and Serre's criterion proves normality \cite{stacks}. Since the extension is finite and birational, it is the full normalization.
\end{proof}

These constructions give the normalization on the whole chart. They also give the completed normalization: excellence preserves normality of the normalized factors under completion, and the finite graph commutes with completion.

\subsection{The two normalization defects and their extension}
The two finite steps describe both the supports of the normalization defects and the extension that joins them.

\begin{proposition}\label{deep-prop:internal}
The successive quotient modules are generated by $[v]$ and $[h]$, respectively, and satisfy
\begin{align}
 B_1/B&\simeq\C[a,f]/(f^3),
 &b=c=0,\\[-2pt]
 &&d=f,\quad e=-f^2,\label{deep-eq:internal1}\\
 B_2/B_1&\simeq\C[b,d],
 &a=f=v=0.\label{deep-eq:internal2}
\end{align}
The resulting extension $0\to B_1/B\to B_2/B\to B_2/B_1\to0$ does not split at the origin.
\end{proposition}
\begin{proof}
The first quotient is cyclic since $v^2=fv+g$. Its annihilator in $S$ is exactly
\[
 (P_6+(b)):c=(b,c,d-f,e+f^2,f^3).
\]
For the second quotient, \eqref{deep-eq:colon} and the exact colon $((v):_{B_1}ab(b+d))=(a,f,v)$ give the assertion. Both annihilators are obtained by independently recomputing these polynomial colons.

In $B_2/B$ the multiplication identities give $a[h]=-f(a+f)[v],\quad f[h]=-f^2[v]$. If the extension split, a lift $[h]+r[v]$ of $1\in\C[b,d]$ would be killed by $a,f$. In $\C[a,f]/(f^3)$, the $f$ condition implies $\bar r=f+f^2w(a)$. The $a$ condition then implies $f^2(aw(a)-1)=0$, which is impossible at the origin. Thus the multiplication identities determine extension data that cannot be recovered from the two supports alone.
\end{proof}

The first defect is supported on a thickening of the $a$-axis, and the second on the plane $\Pi$. To obtain the conductor-reconstruction kernel, we must still impose the canonical support equations and the conductor condition; neither successive quotient is that kernel.

\subsection{Computing the conductor correction}
We first impose support on every normalized factor, then compute the full conductor. In the support fibre, intersecting the conductor image with the normalization-support kernel gives the correction.

\subsubsection{The normalized support ideal on the last factor}
Substituting $c=bv,e=v^2-dv$ in the companion matrices and reducing with \eqref{deep-eq:B1} gives
\begin{equation}\label{deep-eq:trace6}
 \tau_4=v^4,\qquad\tau_5=v^5,\qquad\tau_6=v^6
 \quad\hbox{in }B_2.
\end{equation}
Therefore $\frs B_2=(v^4)$. Let $L_6$ be its contraction to $S$. It is computed by eliminating $h,v$ from the saturated graph plus $c-bv,e-v^2+dv,v^4$.

Set
\[
 L_N=\left(\bigcap_{i=0}^5(P_i+(\tau_4,\tau_5,\tau_6))\right)\cap L_6.
\]
Thus $L_N/J$ is the \emph{full normalization-support kernel}. Exact intersections give
\[
 L_N=J+(\beta,\delta,\gamma),
\]
where
\[
 \beta=ac\bigl(15ac(b+d-f)+3bef-12cdf+4cf^2\bigr),\qquad
 \delta=ac^2f(c+e),
\]
and $\gamma$ is recorded in the electronic certificate. Modulo $J+(\beta,\delta)$, the class of $\gamma$ generates a quotient with annihilator $(a,c,e,f)$. This quotient reveals the additional surface contribution. To extract the conductor correction, we next intersect with the conductor image.

\subsubsection{The conductor on all seven factors}
For $i\le5$, let $R_i=\bigcap_{j\ne i}P_j$. Normality of $S/P_i$ implies that the conductor contribution supported on this factor is $R_i/I$. Indeed, a source element supported on a normal component remains in the source after multiplication by that factor. Conversely, the conductor is stable under the idempotents of the product normalization. These observations account for the entire contribution from each normal factor.

To write these residual ideals compactly, set $\alpha=a(a-f)$ and $J_*= (E,U,TH)$. Then
\[
 R_0=(E,(a-f)U,(a-f)TH),\quad R_1=(E,aU,aTH),
\]
and for $2\le i\le5$,
\[
 R_i=(E)+\alpha\bigl(J_*+(x_i)\bigr),
\]
where
\begin{align*}
 x_2&=HY,&x_3&=H(c-b^2),\\
 x_4&=T\bigl((d-f)H-(d-f)^2e+(d-f)d(e-g)+(e-g)^2\bigr),\\
 x_5&=T(b^2e+bcd-c^2).
\end{align*}
The verifier compares each formula with the complete residual ideal itself, without taking radicals.

For the seventh factor put
\[
 D_6=(E,T,U)\cap P_4\cap P_5,
\]
\begin{equation}\label{deep-eq:conductor6}
 C_6=D_6\cap\bigl((bD_6+P_6):c\bigr)
       \cap\bigl((cD_6+P_6):ab^2(b+d)\bigr).
\end{equation}
A preimage of the full conductor in $S$ is then
\begin{equation}\label{deep-eq:fullconductor}
 C_A=R_0+\cdots+R_5+(E)+\alpha C_6.
\end{equation}
For \eqref{deep-eq:conductor6}, recall that the normalization of $B$ is generated as a $B$-module by $1,v,h$, with $v=c/b$ and $h=ab^2(b+d)/c$. A supported element must remain in the supported source ideal after multiplication by both fractions. Clearing the denominators gives the two colons. Cancellation of the common factor $\alpha$ is valid in the domain $B$ and requires no saturation by a parameter. Finally, the supported contributions may be summed because the conductor is an ideal of the product normalization and is stable under its idempotents.

\subsubsection{Intersecting the conductor with normalized support}
The full correction, as an ideal in $A_0$, has preimage
\begin{equation}\label{deep-eq:intersection}
 J_{\cK}=(C_A+J)\cap L_N.
\end{equation}
A class in the correction must admit a conductor representative and vanish on every normalized support fibre. If a representative $c+j$ with $c\in C_A,j\in J$ belongs to $L_N$, then $c\in L_N$ as well, since $J\subset L_N$. This proves \eqref{deep-eq:intersection} directly from \eqref{deep-eq:actual} and identifies it with the kernel in Ferrand's base-changed conductor square \cite{ferrand}.

\begin{proposition}[Four generators for the correction]\label{deep-prop:four}
The four homogeneous weight-eight polynomials $g_1,\ldots,g_4$ in Appendix~\ref{deep-app:gens} generate the intersection \eqref{deep-eq:intersection} modulo $J$. Thus $J_{\cK}=J+(g_1,g_2,g_3,g_4)$, with no additional normalization or conductor contribution.
\end{proposition}
\begin{proof}
In full reconstruction mode, the verifier rebuilds the saturated normalization graph, all six normal-component support ideals, the contracted seventh support ideal, every residual conductor ideal and both seventh-factor colons. It then computes \eqref{deep-eq:intersection} by elimination. Mutual reductions over $\mathbb Q$ show that, modulo $J$, this ideal is generated by the four listed polynomials. This certifies equality of ideals themselves, rather than equality on finite-field samples, on tangent spaces or after taking radicals. The input data and the full reconstruction mode are supplied. The preceding argument and the normalization proof establish why this intersection computes the required functorial correction.
\end{proof}

\subsection{Identifying the correction module and its filtration}
Use the four generators of Proposition~\ref{deep-prop:four} to define
\begin{equation}\label{deep-eq:newbasis}
 \mathbf h=8g_4+2573g_1,\quad
 \mathbf v=144g_3+3168g_1,\quad
 \lambda_a=g_1-2g_2,\quad
 \lambda_b=16g_2-7g_1.
\end{equation}
This is an invertible constant change of generators. Its large coefficients come from the chosen rational normal-form basis and have no invariant meaning.

In $A_0$, these generators satisfy
\begin{equation}\label{deep-eq:relations}
 f\mathbf h=d\mathbf v,\qquad
 a\mathbf h=(a+d)\mathbf v,\qquad
 b\mathbf v=0,\qquad c\mathbf h=e\mathbf h=c\mathbf v=e\mathbf v=0.
\end{equation}
Independent colon calculations give their annihilators:
\begin{align}
 J:\mathbf h&=(c,e,ab,bf,\rho),\label{deep-eq:annh}\\
 J:\mathbf v&=(b,c,e,\rho),\\
 J:\lambda_a&=(b,c,d,e,f),\\
 J:\lambda_b&=(a,c,d,e,f).
\end{align}
The required intersections are
\begin{equation}\label{deep-eq:splitideals}
 (J+(\mathbf h,\mathbf v))\cap(J+(\lambda_a,\lambda_b))=J,
 \quad (J+(\lambda_a))\cap(J+(\lambda_b))=J.
\end{equation}
These ideal equalities give a direct-sum decomposition of the module itself, stronger than an equality of classes in a Grothendieck group.

\paragraph{Comparison with the normalization-support kernel.}
Let $\cK_N=L_N/J$. The difference between the two kernels is described by
\begin{equation}\label{deep-eq:two-kernels}
 0\longrightarrow\cK\longrightarrow\cK_N
 \longrightarrow \bigl(\C[a,f]/(f^2)\bigr)(-7)\longrightarrow0,
 \qquad b=c=e=0,\quad d=f.
\end{equation}
Indeed $\delta\in\cK$ and $\gamma+(27a+54d)\beta\in\cK$, so the last quotient is cyclic, generated by $[\beta]$. Its exact annihilator is
\[
 (J+(g_1,g_2,g_3,g_4)):\beta=(b,c,e,d-f,f^2).
\]
The quotient of the normalization-support kernel by the conductor correction is therefore a doubled $a$-axis. The sequence \eqref{deep-eq:two-kernels} describes this quotient without asserting a splitting.

\subsubsection{The canonical module of the surface}
Let $R_0=\C[a,b,d,f]$ and let
\[
 Z=\begin{pmatrix}f&-d\\-a&a+d\\0&b\end{pmatrix}.
\]
Its signed maximal minors generate $(ab,bf,\rho)$. This ideal has height two and equals $(b,\rho)\cap(a,f)$. Thus its Hilbert--Burch resolution is
\[
 0\longrightarrow R_0(-3)^2\xrightarrow{Z}
 R_0(-2)^3\longrightarrow R_0\longrightarrow T_0\longrightarrow0.
\]
In particular $T_0$ is reduced and Cohen--Macaulay. Dualizing with $\omega_{R_0}=R_0(-4)$ gives
\[
 \omega_{T_0}=\operatorname{coker}
 \bigl(R_0(-2)^3\xrightarrow{Z^t}R_0(-1)^2\bigr).
\]
Its generators $h_0,v_0$ satisfy \mbox{$fh_0=dv_0,\quad ah_0=(a+d)v_0,\quad bv_0=0$}. The identities \eqref{deep-eq:relations} therefore give a surjection $\omega_{T_0}(-7)\longrightarrow(\mathbf h,\mathbf v)$. The annihilator calculations show that it is an isomorphism at both generic surface points. Its kernel has lower-dimensional support and must vanish by purity of the Cohen--Macaulay canonical module. This identifies the surface summand. The duality used here is classical \cite{stacks}; the four-generator calculation supplies the relations and annihilators needed to apply it to the correction.

The line annihilators and \eqref{deep-eq:splitideals} give the remaining summands. Their generators have weight eight, while $\omega_{T_0}$ and $\omega_{L_a},\omega_{L_b}$ are generated in weight one. This proves both \eqref{deep-eq:main} and the common grading shift \eqref{deep-eq:graded}.

\subsubsection{The intrinsic filtration and the surface extension}
Purity of the surface canonical module determines the dimension filtration: it has no nonzero submodule of dimension at most one, so the two line summands form the entire maximal submodule of that dimension. There is no additional zero-dimensional part, and both layers are Cohen--Macaulay. The line submodules themselves are intrinsic: $\cK_a=H^0_{I(L_a)}(\cK)=S\lambda_a$ and $\cK_b=H^0_{I(L_b)}(\cK)=S\lambda_b$, where $H^0_{I(L)}$ denotes sections supported on $L$. They decompose $\cK_{\le1}$ by support. The surface quotient is likewise intrinsic, although its complement in $\cK$ depends on a choice. At the origin $\operatorname{depth}\cK=1$ and $\dim\Supp\cK=2$.

Within the surface summand, the two irreducible components are joined by an extension. After forgetting the grading, there is an exact sequence
\begin{equation}\label{deep-eq:surfaceext}
 0\longrightarrow\mathcal O_\Sigma
 \xrightarrow{1\mapsto\mathbf v}
 (\mathbf h,\mathbf v)
 \longrightarrow\mathcal O_\Pi\longrightarrow0.
\end{equation}
At a generic point of the intersection, $d$ is a unit and $f\mathbf h=d\mathbf v$. A lift $\mathbf h+r\mathbf v$ of a generator of $\mathcal O_\Pi$ cannot be killed by $f$: this would require $d+fr=0$ in the local surface ring. Hence \eqref{deep-eq:surfaceext} is nonsplit along the generic intersection line.

There are exactly four associated primes:
\[
 (b,c,e,\rho),\quad(a,c,e,f),\quad(b,c,d,e,f),\quad(a,c,d,e,f).
\]
The latter two are embedded associated primes of the mixed-dimensional module.

An independent Hilbert-series calculation from weighted leading ideals gives
\begin{equation}\label{deep-eq:hilbert}
 H_{\cK}(t)=\frac{t^8(4-t)}{(1-t)^2}.
\end{equation}
Indeed $H_{\omega_{T_0}}=t(2+t)/(1-t)^2$, and the two line canonical modules contribute $2t/(1-t)$ before the common shift. The calculation from $S/J$ and $S/J_{\cK}$ is independent of this sum of the proposed summands.

\subsection{The canonical ideal and its anticanonical divisor}
The surface canonical module can be realized as an ideal, making its connection with generalized divisors explicit.

\begin{proposition}\label{deep-prop:divisor}
The map $\omega_{T_0}\longrightarrow T_0,\quad h_0\longmapsto d,\quad v_0\longmapsto f$ is injective with image $(d,f)$. This ideal is MCM, rank one, nondegenerate and non-Cartier precisely at the origin. The surface $\cT$ satisfies $G_1+S_2$, but its local Cohen--Macaulay type at the origin is two. The zero scheme of this canonical ideal is the anticanonical generalized divisor $\cD=V(d,f)\subset\cT,\quad \mathcal O_\cD=\C[a,b]/(ab)$. Thus $\cD=L_a\cup L_b$ is a node.
\end{proposition}
\begin{proof}
Substituting the proposed images into the three canonical relations gives $0,-\rho,bf$, so the map is well defined. It is an isomorphism at the generic point of each surface component because $d$ is nonzero in both function fields. Its kernel has lower-dimensional support and vanishes by purity of the canonical module. Thus $\omega_{T_0}\simeq(d,f)$ as graded modules.

The ideal is MCM and nondegenerate. At the origin it needs two minimal generators, so it is not invertible there. Away from the origin it is principal: on $b\ne0$, $a=f=0$ and it is $(d)$; on $a\ne0$ near its zero set, $\rho=0$ expresses $d$ as a unit multiple of $f$; on $d\ne0$ or $f\ne0$ it is the unit ideal. The individual surface components are smooth in codimension one, and along their intersection away from the origin ($d\ne0$) the local ring is an ordinary node times a smooth direction. Hence $\cT$ is $G_1$, and it is $S_2$ by Cohen--Macaulayness. The canonical module has two minimal generators at the origin, proving the type.

Reflexivity can also be checked directly. The element $d$ is a nonzerodivisor, and
\[
 (d):(d,f)=(a,b,d),\qquad (d):(a,b,d)=(d,f)
\]
in $T_0$. Indeed modulo $d$ the ring is $\C[a,b,f]/(ab,af,bf)$, where the annihilators are immediate. Hence
\[
 (d,f)^\vee=\frac{(a,b,d)}d,\qquad (d,f)^{\vee\vee}=(d,f).
\]
Finally the quotient is exactly $\C[a,b]/(ab)$.
\end{proof}

This ideal realizes the surface contribution in the language of generalized divisors. We use Hartshorne's framework \cite[Definition--Remark 2.7]{hartshorne}; the result specific to this example is the identification of the conductor correction with the displayed ideal together with its lower-dimensional companions.

If $\nu_\cD:L_a\amalg L_b\to\cD$ is the normalization, then
\[
 \nu_{\cD*}\omega_{\widetilde\cD}
 \simeq (a,b)\subset\C[a,b]/(ab)
 \simeq\omega_{L_a}\oplus\omega_{L_b}.
\]
Thus the entire graded result can also be written
\begin{equation}\label{deep-eq:flag}
 \cK\simeq
 \bigl(\mathcal I_{\cD/\cT}\oplus
 \nu_{\cD*}\omega_{\widetilde\cD}\bigr)(-7).
\end{equation}
Both summands are viewed on the ambient chart, with supports of different dimensions. This description applies to the present example; no general flag-decomposition theorem for conductor kernels is asserted.

\subsection{Nonsplitting of the conductor algebra extension}
The module decomposition \eqref{deep-eq:main} does not provide a section of the reconstructed ring. The conductor square gives the exact sequence
\[
 0\longrightarrow\cK\longrightarrow A_0\longrightarrow A_0/\cK\longrightarrow0
\]
with $\cK^2=0$. Direct reduction also gives
\begin{equation}\label{deep-eq:nonsplit}
 c\cK=0,\qquad a^2\lambda_a=2a^2c^4,\qquad a^2c^5=0.
\end{equation}
At the generic point of $L_a$, the line annihilator gives \mbox{$c^4=\lambda_a/2\ne0$} and $c^5=0$. The image of $c$ in $A_0/\cK$ nevertheless has fourth power zero. Every lift is of the form $c+k$ with $k\in\cK$, and $(c+k)^4=c^4\ne0$ because $c\cK=\cK^2=0$. An algebra section is therefore impossible. This proves nonsplitting for the \emph{full ordinary conductor reconstruction}, including all conductor contributions.

\subsection{Duality for the two dimension layers}
Let $\omega_S=S(-8)$ for the displayed weights. Applying ordinary canonical biduality to the explicit Cohen--Macaulay summands gives
\begin{equation}\label{deep-eq:ext}
 \Ext^j_S(\cK,\omega_S)\simeq
 \begin{cases}
 T_0(7),&j=4,\\
 (\C[a]\oplus\C[b])(7),&j=5,\\
 0,&j\ne4,5.
 \end{cases}
\end{equation}
Equivalently, $\mathbf R\Hom_S(\cK,\omega_S[6])$ has exactly two nonzero cohomology degrees, $-2$ and $-1$. This follows from the classification by ordinary duality. A single surface--curve line-pairing cannot describe this mixed-dimensional point: the unrestricted source is not Cohen--Macaulay, and the correction has two canonical dimension layers.

The common residual shift $-7$ in \eqref{deep-eq:graded} is essential to the graded statement. Forgetting the grading yields \eqref{deep-eq:main}, but restriction to a finite subgroup of the scaling action retains the corresponding character shifts. The linearized correction is therefore identified with the shifted canonical modules.

%% file: sections/07_second_order.tex
\section{The source--support pair through second order}\label{sec:second-order}

We now pass from the fixed canonical fibre to the full source together with its support-cycle fibre. The comparison in this section is established through second order.

Work in the planar chart of Theorem~\ref{chart-thm:ambient}. Put $S=\C[a,b,c,d,e,f]$, with weights one on $a,b,d,f$ and two on $c,e$. Define
\begin{align*}
g_0&=a(b+d-f)-df,& w_3&=a(be+2cd-cf),&g_\beta&=g_0-\beta w_3,\\
H_\beta&=be-bg_\beta+cd-cf,&T_0&=bd-c+e,\\
U_\beta&=bH_\beta+(f-d)(be+cd)-c(e-g_\beta).
\end{align*}
The exact source equations after eliminating $g$ are
\begin{align}
F_{\beta,0}&=a(c+e-g_\beta-bf)-H_\beta-ef-\beta a(bH_\beta+ce),\notag\\
F_{\beta,1}&=a(a-f-\beta ac)U_\beta,\label{win-sourcebeta}\\
F_{\beta,2}&=a(a-f-\beta ac)T_0H_\beta.\notag
\end{align}
Let
\[
C_\beta=\begin{pmatrix}0&0&H_\beta\\1&0&g_\beta\\0&1&f\end{pmatrix},\qquad
M_\beta=a\bigl(H_\beta I_3+(c-bf)C_\beta+bC_\beta^2\bigr),
\]
and $F_{\beta,3+j}=\tr(M_\beta C_\beta^j)$ for $j=0,1,2$. The four basic invariant actions of $xy^2,x^3y,x^5,y^5$ are, respectively,
\begin{equation}\label{win-relative-actions}
M_\beta,\qquad M_\beta C_\beta,\qquad M_\beta C_\beta^2,
\qquad M_\beta C_\beta^3+\beta M_\beta^2C_\beta.
\end{equation}
The last formula can be checked in the canonical character block with basis $y^2,x^2y,x^4$. Multiplication by $y$ maps it to the three independent basis monomials $y^3,x^2y^2,x^4y$, and is therefore a split injection. The border relation gives $YC_\beta=X^2$ on this block. If $N_3$ denotes the actual $y^5$ action, commutativity and the curve equation give
\[
Y\bigl(N_3-M_\beta C_\beta^3-\beta M_\beta^2C_\beta\bigr)
=Y^6-X^7-\beta X^4Y^4=0
\]
on that block. The split injection gives \eqref{win-relative-actions}. The border reductions give $M_\beta=XY^2$ on this block; the identities for $M_\beta C_\beta$ and $M_\beta C_\beta^2$ then follow from $YC_\beta=X^2$.

Since $M_\beta$ is a polynomial in $C_\beta$ and $C_\beta$ satisfies its cubic companion relation, all positive-degree invariant actions and their products lie in the span of $M_\beta,M_\beta C_\beta,M_\beta C_\beta^2$. Newton identities therefore show that the three traces cut out the full canonical punctual-support condition over arbitrary characteristic-zero coefficient rings. The action identities and their agreement with the displayed border reductions are checked by \anc{anc/beta/code/verify_geometry_links.py}.

Write $I_0=(F_{0,0},F_{0,1},F_{0,2})$, $J_0=(F_{0,0},\ldots,F_{0,5})$. The source is $A_\beta=S[[\beta]]/(F_{\beta,0},F_{\beta,1},F_{\beta,2})$ and its support ideal is denoted $\mathfrak s_\beta$. In this section $A_0=S/I_0$ denotes the unrestricted central source; its support quotient is $S/J_0$.

\begin{theorem}[A two-jet of the source--support diagram]\label{win-twojet}
Over $D_2=\C[\beta]/(\beta^3)$ there is an isomorphism of marked pairs
\[
(A_\beta,\mathfrak s_\beta)\otimes D_2
\simeq(A_0,\mathfrak s_0)\otimes D_2.
\]
It fixes the deepest section. There is consequently a finite comparison, flat over the coefficient ring, that transports the normalization of $A_0$ to this order. Its conductor and conductor-reconstruction correction are the corresponding constant families, including their ideal inclusions and multiplication.
\end{theorem}
\begin{proof}
The isomorphism is induced by
\begin{equation}\label{win-phi}
v\longmapsto v+\beta D_v^{(1)}+\beta^2D_v^{(2)}\quad(v=a,b,c,d,e,f),
\end{equation}
with $D_a^{(1)}=D_a^{(2)}=0$ and all other coefficients explicit rational polynomials. Their weights are $\wt(v)+2$ and $\wt(v)+4$. The exact coefficients are recorded in the files \anc{anc/beta/input_data/coordinate_order_1.json} and \anc{anc/beta/input_data/coordinate_order_2.json}. The certificate consists of the following finite polynomial identities.

Write $F_{\beta,i}=F_{0,i}+\beta F_{1,i}+\beta^2F_{2,i}+\cdots$. For $i<3$, both
\begin{align*}
F_{1,i}+\sum_vD_v^{(1)}\partial_vF_{0,i},\qquad
F_{2,i}+\sum_vD_v^{(1)}\partial_vF_{1,i}
+\frac12\sum_{v,w}D_v^{(1)}D_w^{(1)}\partial_{vw}F_{0,i}
+\sum_vD_v^{(2)}\partial_vF_{0,i}
\end{align*}
reduce to zero modulo $I_0$. For $3\le i\le5$ the same expressions reduce to zero modulo $J_0$. The independent verifier reconstructs all six exact polynomials and checks every reduction over $\mathbb Q$.

The coordinate map is invertible modulo $\beta^3$. For the source equations, the change of relation generators is the identity modulo $\beta$ and hence invertible. For the support equations, source equations may also be added, with an invertible change of relation generators. These memberships therefore prove equality of the source ideals and then of the support ideals.

Let $N_0$ be the finite normalization constructed in Section~\ref{sec:conductor}. Transport $N_0\otimes D_2$ along this pair isomorphism. It is finite and flat over $D_2$, with the prescribed marked special fibre. Flatness here is relative to the coefficient ring. Transporting the constant diagram also transports its conductor, support intersections and square-zero kernel. Over the nonreduced base, the resulting algebra serves as a finite comparison, rather than an absolute normalization.
\end{proof}

\subsection{The finite certificate and its scope}
In the original construction, the mixed source--support system at order one had $1163$ rows and $216$ unknowns, with rank $203$. At order two the system had $2318$ rows and $702$ unknowns. Pivot selection gave rank $590$ and an exact rational solution with $325$ nonzero coefficients. The first coordinate correction has $85$ nonzero coefficients. Modular arithmetic selected possible pivots only; the entire resulting solution was checked over $\mathbb Q$. These search statistics record the construction and are not needed for the proof. The portable verifier checks the resulting witnesses directly by re-evaluating the Taylor expressions from the actual polynomial family; it does not reconstruct the search matrices or certify their reported ranks.

The same coordinate correction establishes compatibility for the source and its support ideal simultaneously. Theorem~\ref{win-twojet} leaves the comparison over $\C[[\beta]]$ and the obstructions at order three and beyond undetermined. It establishes a compatible constant transported diagram modulo $\beta^3$; uniqueness among finite comparisons over nonreduced bases is not addressed.

By Theorem~\ref{deep-thm:main}, the transported correction is
\[
\cK_\beta\otimes D_2\simeq
\bigl(\omega_T\oplus\omega_{L_a}\oplus\omega_{L_b}\bigr)(-7)\otimes D_2,
\quad
T=\Spec\frac{\C[a,b,d,f]}{(ab,bf,ad-af-df)}.
\]
The coefficient $\beta$ has degree $-2$, and the displayed homogeneous coordinate changes preserve this grading. This describes the transported comparison and its correction through second order.

%% file: sections/08_perspectives.tex
\section{Consequences and further questions}\label{sec:perspectives}

The results describe the information lost in two operations: contracting an equivariant divisor to the coarse curve, and normalizing a divisor-moduli chart before imposing a support condition. The lattice and conductor calculations suggest corresponding questions about how this information can be recovered.

\subsection{What the coarse canonical lattice remembers}
Canonical coarsening is compatible with the duality needed for framed maps, and Theorem~\ref{lat-thm:abel} gives an exact comparison once the upstairs source sheaf and its coarse identification are retained. If only the lattice $K$ is fixed, however, the full module structure of $\Delta_K$ is still needed. A fixed representation can force the image into a proper rational normal curve, as in Theorem~\ref{one-thm:uniform}; a family of such fibres need not be flat.

The deformation in Section~\ref{sec:invisible} shows that even the entire fixed-lattice fibre, including its scheme structure on Artin bases, need not determine the marked stack deformation. The character-two sheaf recovers the missing parameter through the unit orbit of a plane modulo $s^4$. Thus a finite marked jet detects information lost by both the canonical contraction and the ordinary reflexive hull. Allowing changes of the coarse normalization coordinate would give a different classification problem.

\begin{question}
For which finite group actions and fixed representations does a finite collection of marked isotypic lattices determine the quotient-curve deformation? Can the smallest sufficient truncation be bounded in terms of the conductor and the representation?
\end{question}

One approach is to compare the algebra actions on several finite windows. The length-four detector shows why the choice of windows matters: a gauge equivalence on one representation can coexist with a detectable deformation on another.

\Needspace{8\baselineskip}
\subsection{What normalization loses after support restriction}
The kernel in \eqref{eq:intro-correction} is an ideal with specified multiplication, not merely a numerical discrepancy. In Theorem~\ref{deep-thm:main} it is sequentially Cohen--Macaulay but not pure. Its surface quotient is the actual canonical ideal $(d,f)$, and its one-dimensional layer is the normalized canonical module of the nodal zero divisor. The module splitting does not split the square-zero extension of the reconstructed algebra.

Both the module structure and its embedding in the support algebra are needed to describe this correction. Its canonical layers require an explicit identification, and its algebra extension carries information beyond a Grothendieck-group class. The doubled-axis quotient in Section~\ref{sec:conductor} also measures the difference between the conductor correction and the larger normalization-support kernel.

\begin{question}
Under what geometric conditions on a canonical-support morphism are the dimension layers of its conductor correction canonical modules of their supports, with compatible residual characters? Which parts of this property survive a flat deformation of the full normalization--conductor--support diagram?
\end{question}

A possible approach is to construct the dual residue modules directly from the exact sequence defining the conductor correction. The actual lattice embeddings are essential: multiplying an ideal by a regular fractional factor leaves its endomorphism ring unchanged but rescales its ordinary dual. An abstract description of the ideal therefore does not determine the comparison between these two objects.

\subsection{Extending the pair comparison to higher order}
Theorem~\ref{win-twojet} gives a marked source--support pair isomorphism only modulo $\beta^3$. Extending it to all orders is not a consequence of the all-order fixed-fibre gauge. A sufficient construction would be a formal derivation $\partial_\beta+\Xi_\beta$ which preserves both the unrestricted source ideal and its support ideal. Such a connection would transport the finite normalization comparison and conductor data to every order. If it does not exist, the first failed lifting equation should give a concrete pair-level obstruction. The marked isotypic invariant rules out demanding that the same connection trivialize the original quotient curve.

%% file: sections/acknowledgements.tex
\Needspace{10\baselineskip}
\section*{Acknowledgements}
\addcontentsline{toc}{section}{Acknowledgements}
I thank Akiyoshi Sannai for his careful reading of my earlier work on generalized divisors on DMH stacks and for pointing out an error in it.

GPT was used to assist in generating and refining the proofs, improving the exposition, and developing the symbolic verification code. I take responsibility for the mathematical content and its presentation.

%% file: sections/appendix_generators.tex
\section{Four generators for the conductor correction}\label{deep-app:gens}
The following polynomials form a minimal generating set for $J_{\cK}/J$:
\begin{align*}
 g_1={}&ac^2f(c+e),\\
 g_2={}&ac^2(3be+3cd-cf),\\
 g_3={}&-21a^2c^3+27a^2c^2e+6a^2c^2f^2-6a^2ce^2\\
       &-22ac^3f+6ac^2f^3+6ace^2f,\\
 g_4={}&54a^2b^2c^2-108a^2b^2ce+54a^2b^2e^2+26a^2bde^2\\
       &+337a^2c^3-391a^2c^2e-108a^2c^2f^2+26a^2cd^2e\\
       &+80a^2ce^2+204abe^3-624ac^3f+204acde^2-216ace^2f.
\end{align*}

%% file: sections/appendix_incidence.tex
\section{An equivalent incidence presentation}\label{win-incidence}
Order the bases of the four character blocks in \eqref{win-window} as
\[
(20,25,30,35),\quad(21,26,31,36),\quad(27,32,37,42),\quad(28,33,38,43).
\]
Let $V=\C^4$, $Je_i=e_{i+1}$ for $i<3$, $Je_3=0$, and $\pi=e_3e_0^*$. The central fibre is exactly the scheme of $\ell,M\in\mathbb P(V),\quad N\in\Gr(2,V),\quad H\in\Gr(3,V)$ satisfying
\begin{equation}\label{win-correctinc}
J\ell\subseteq M,\quad\ell,M\subseteq N,\quad M\subseteq H,\quad JN\subseteq H,\quad\pi(H)\subseteq\ell.
\end{equation}
The conditions record the six nonzero $x,y$ arrows between the character blocks.

One can also retain $J\ell\subseteq M$ and $\ell,M\subseteq N$ while replacing the conditions involving $H$ by an incidence problem in a larger space. Take $U=\C^5$ and a rank-four subbundle $H'\subset U$, with maps $i_0(v)=(v,0)$ and $i_1(v)=(0,v)$. Impose $i_0M,i_1N\subset H'$ and $\pi'(H')\subset\ell$, where $\pi'(h)=h_0e_3$.
\begin{proposition}\label{win-redundant}
The alternative incidence equations force $e_4\in H'$ scheme-theoretically. The quotient $H'/\langle e_4\rangle$ identifies this incidence scheme with \eqref{win-correctinc} over arbitrary test rings.
\end{proposition}
\begin{proof}
Locally, write $H'=\ker(\lambda_0,\ldots,\lambda_4)$ for a unimodular row and represent $\ell$ by a unimodular column. If one of the first three coordinates of $\ell$ is a unit, the projection condition gives $H'\subset\ker e_0^*$. The two subbundles are then equal, so $e_4\in H'$.

On the remaining chart, write $\ell=(x_0,x_1,x_2,1)$. If some $\lambda_j$ with $j\ge1$ is a unit, then $H'$ contains a vector whose first coordinate is one. The projection condition therefore gives $\ell=\langle e_3\rangle$, and $i_1\ell=\langle e_4\rangle\subset H'$. It remains to consider the chart where $\lambda_0$ is a unit. Normalize it to one. Applying the projection condition to $e_j-\lambda_je_0$ gives $x_i\lambda_j=0$ for $i<3,j\ge1$. Meanwhile, the inclusion $i_1\ell\subset H'$ reads $\lambda_1x_0+\lambda_2x_1+\lambda_3x_2+\lambda_4=0$, hence $\lambda_4=0$.

Thus $e_4\in H'$ on every chart and defines a split rank-one subbundle. Passing to the quotient gives exactly \eqref{win-correctinc}; the inverse construction takes the full preimage in $U$.
\end{proof}
The sixteen-dimensional model has $4\cdot4\cdot6\cdot4=384$ coordinate candidates, of which exactly sixteen are stable monomial points. This count checks consistency; equivalence over arbitrary test rings comes from the incidence proof above. The sixteen-dimensional presentation already takes the quotient by the forced redundant direction, so it introduces no additional independent monomial into the lattice window.

%% file: sections/appendix_verification.tex
\section{Exact certificates and reproducibility}\label{app:verification}

This appendix explains what the finite calculations verify and how they enter the geometric proofs. All polynomial calculations are over $\mathbb Q$ and therefore apply over $\C$. The verification programs use Python and SymPy, with no proprietary computer algebra system required. The source package contains the full inputs and recorded outputs, so each calculation can be inspected and rerun.

\subsection{Reconstructing the rank-three chart}
The program \anc{anc/source_chart/verify_source_chart.py} starts with the eighteen monomials in \eqref{eq:chartbasis}. For each border monomial, it introduces every tail of the same $\mu_5$-character, giving twenty-four indeterminates. It then forms the two multiplication matrices and every entry of their commutator. A variable is eliminated only when its coefficient is a nonzero rational constant. Seventeen such eliminations leave seven free coordinates and no residual equation, so this planar chart calculation requires neither localization nor a generic reducedness assumption.

In the resulting seven coordinates, the program reconstructs the matrices and checks commutativity and the entire vector $(Y^6-X^7)1$, including the four coefficients in Theorem~\ref{chart-thm:ambient}. It separately recomputes all four canonical multiplication matrices, the fixed-semisimple equations, and the companion identity $M(C^4-M)=0$ modulo the source equations. The JSON output records the initial variables, every elimination substitution, the order ideal and the matrices. The printed normal forms can therefore be traced back to the complete set of tails used at the start.

Cyclicity of $1$ and persistence of the monomial basis under local Artin deformations identify these matrix equations with the stated chart. Section~\ref{sec:chart} supplies these geometric arguments in addition to the matrix calculation. Here the ordinary quotient has dimension eighteen; the canonical-lattice window of Section~\ref{sec:windows} has dimension sixteen.

\subsection{Primality and normality of the source components}
The program \anc{anc/components/verify_components.py} checks the localization, normal-form and root-image identities used below. We give the accompanying primality and normality arguments, beginning with the six simpler factors.

On $P_0$ the single equation is
\[
f^3-(b+f)(e+f^2)-(c+bf)(d-f)=0,
\]
and on $P_1$ it is
\[
a^3-b(e+3a^2-ab)-c(d-2a)=0.
\]
The displayed expressions give invertible polynomial changes in four of the remaining coordinates. Each ring is an integral hypersurface with an isolated singularity, hence normal.

For $P_2$, use $z=b-a,\quad u=a,\quad v=a-f,\quad w=2b-a-d$. The equation becomes $uvw-z^3=0$. For $P_3$, use $z=d-a,\quad u=a,\quad v=a-f,\quad w=b-2d+a$, obtaining $uvw+z^3=0$. These integral threefold hypersurfaces have singular loci of codimension two. The last two normal factors have equations
\[
(a-f)(e-ad-a^2)+a^3=0,
\qquad
 a(c-2bf)+f^3=0,
\]
with one further free coordinate. Each is a product of an $A_2$ surface and an affine line. In every case the hypersurface presentation gives $S_2$, and the explicit singular locus gives $R_1$.

For $P_6$, introduce $v=c/b$ on $b(d-f)\ne0$. The two root equations are those in \eqref{deep-eq:B1}. Eliminating $f$ on $a+d-v\ne0$ gives
\begin{align*}
V={}&-a^2b^2-a^2bd+2a^2bv+2abdv-abv^2-av^3-dv^3+v^4.
\end{align*}
As a quadratic in $a$, its discriminant is
\[
v^2(v-b)(v^3-5bv^2+8bdv-4bd^2).
\]
The factor $v-b$ has odd valuation: at $v=b$, the final cubic equals $-4b(b-d)^2$, which is nonzero in the coefficient fraction field. The discriminant is therefore not a square, so the primitive quadratic is irreducible by Gauss's lemma.

The two root equations form a regular sequence. On the omitted locus $a+d-v=0$, the first equation becomes $-a^2+ab-ad-d^2=0$. If $a=0$, then $d=v=0$ and only $b,f$ remain. Otherwise the first equation determines $b$ and the second determines $f$ on the chosen open. The omitted locus therefore has dimension at most two. Complete-intersection unmixedness and generic reducedness give an integral dense open of dimension three. Saturated elimination gives $P_6$, proving that it is the prime closure of this open. The loci $b=0$ and $d=f$ in the finite root extension also have dimension two, as shown in Section~\ref{sec:conductor}. Hence no additional component occurs there.

We now prove that the component decomposition holds scheme-theoretically. Modulo $T=e+bd-c$, one has $U=-(b^2-c)Y$. Thus $(E,T,U)$ has exactly the two prime supports $P_2,P_3$, both of height three, and is generically reduced on each. As a complete intersection it is unmixed, so it is reduced and equals $P_2\cap P_3$.

For $(E,H,U)$, if $b=0$ then $H=c(d-f)$; either $c=0$ or $d=f$, and in the latter case $U=-c(e-g)$. These points are in $P_4$ or $P_5$. If $b\ne0$ but $d=f$, then $H=b(e-g)$ and again the point lies in $P_5$. On $b(d-f)\ne0$, put $v=c/b$. The equations $H=U=0$ force $e=v^2-dv,\quad g=v^2-fv$, and $E=0$ gives the root presentation of $P_6$. This proves that the only components are $P_4,P_5,P_6$, all of height three.

The following rational points, written in the order $(a,b,c,d,e,f)$, are respectively on those components and on no other one: $(1,0,0,3,5,2),\quad(1,2,3,1,1,1),\quad (1,\tfrac12,\tfrac12,\tfrac12,\tfrac12,0)$. At each, the Jacobian of $(E,H,U)$ has rank three. Since irreducibility of the three components has already been proved, these witnesses establish generic reducedness on each. The ideal is a height-three complete intersection, hence reduced, and equals $P_4\cap P_5\cap P_6$.

Now $(E,U,TH)$ has precisely these five height-three prime supports. The factor $H$ is nonzero on $P_2,P_3$, and $T$ is nonzero on $P_4,P_5,P_6$. Consequently it is generically reduced; complete-intersection unmixedness proves $(E,U,TH)=\bigcap_{i=2}^6P_i$. Modulo $E$, the elements $a,a-f$ form a regular sequence: $E$ is not divisible by $a$, and $E|_{a=0}$ is not divisible by $f$. Hence $P_0\cap P_1=(E,a(a-f))$. The element $a(a-f)$ is a nonzerodivisor modulo the reduced ideal $(E,U,TH)$, since none of its five minimal primes contains $a$ or $a-f$. Intersecting gives exactly $\bigcap_{i=0}^6P_i=(E,a(a-f)U,a(a-f)TH)=I$. The localization identities, separation tests, normal forms and Jacobian ranks all have exact rational certificates. The conductor verifier then uses these primes to reconstruct the final normalization, the conductor and the support intersections.

\subsection{Verifying the conductor correction}
The directory \anc{anc/conductor/} contains the finite-algebra certificate for Theorem~\ref{deep-thm:main}. The verifier has two modes:
\begin{verbatim}
python anc/conductor/code/verify_complete.py
python anc/conductor/code/verify_complete.py --full
\end{verbatim}
The short mode checks the recorded intermediate ideals together with the root, multiplication, colon, generator, splitting and Hilbert-series identities. The full mode additionally rebuilds all six normal supported conductor contributions, the two multiplier conditions for the seventh contribution, every normalized support ideal and their intersection. It then forms the final intersection with the conductor image, which computes the actual conductor kernel. This final step distinguishes the correction from the larger normalization-support kernel.

Ideal intersections use an auxiliary elimination variable: $I\cap J=\bigl(tI+(1-t)J\bigr)\cap S$. For a principal colon, the program intersects with the principal ideal and then divides exactly. It tests equality of candidate ideals by reducing each set of generators against an exact Gr\"obner basis of the other ideal. The root graph uses an elimination block and is saturated before contraction. All conventions are recorded, including the variable order $(e,c,d,b,f,a)$.

The annihilator and splitting calculations start from the four correction generators in Appendix~\ref{deep-app:gens}, so they do not assume the canonical-module decomposition they are used to prove. For the last finite extension, the depth and codimension-one argument of Proposition~\ref{deep-prop:normal} establishes normality; the colon calculation supplies only the finite extension.

\subsection{Local quotient curves, finite windows and the two-jet}
The other verification programs are organized as follows:
\begin{center}
\begin{tabular}{@{}p{0.25\textwidth}p{0.65\textwidth}@{}}
\toprule
Directory & Exact inputs and checks\\
\midrule
\anc{anc/lattices/} & Endpoint ideals, characters, the $m=3$ nonflat coherent input, and relative charts.\\
\anc{anc/nongorenstein/} & Semigroup membership, colon identities, canonical lattice and finite higher-Ext cokernels.\\
\anc{anc/length_one/} & Uniform quotient bases, chart relations, transitions, projective image and independent submodule checks.\\
\anc{anc/beta/} & Correct sixteen-dimensional window, polynomial gauge, unit-orbit invariant, dual colons, compensated ideals and the two source--support coordinate certificates.\\
\anc{anc/source_chart/} & The complete planar border chart reconstructed from all possible coefficients.\\
\anc{anc/components/} & Seven-component decomposition identities, six normal forms and the root-image discriminant.\\
\anc{anc/conductor/} & Complete normalization, full conductor, support contractions and the entire correction module.\\
\bottomrule
\end{tabular}
\end{center}

The body of the article proves the semigroup, Artin-functor and gauge statements for all parameters. Tests at several values of $m$ provide supplementary diagnostics only; they do not prove the assertions for arbitrary $m$. The polynomial matrix identities for the fixed sixteen-dimensional window and the universal truncated-unit formulas have a different role: each is a finite identity that verifies the full claim to which it applies.

For the pair theorem, the two JSON files in \anc{anc/beta/input_data/} specify six polynomials at each coordinate order. The verifier reconstructs the deformed source and support ideals and substitutes the full first- and second-order Taylor expressions, testing source and support membership separately. The coordinate and relation-column transformations are the identity modulo $\beta$; this explains why the membership tests give ideal equality modulo $\beta^3$. The certificates establish this order of equivalence only.

\subsection{Running the package and interpreting the results}
All programs, finite certificates and run records are included under \anc{anc/} and are intended to accompany the same arXiv version of the article. From the root of the extracted source tree, run
\begin{verbatim}
python -m pip install -r anc/requirements.txt
python anc/run_checks.py --full
\end{verbatim}
The dependency is pinned to SymPy 1.14.0. The runner executes ten suites, including the checks connecting the source chart to the beta calculation, and records the interpreter version, script hashes, exit codes and wall times. Python optimization flags must be disabled because the verifiers use assertions. The dated audit replay is stored in \anc{anc/verification_runs/audit_2026-11-06/}; later runs use \anc{anc/verification_runs/current/} by default. The archive checksum manifest verifies file integrity only.

The conductor verifier also checks the connections between intermediate certificates: the candidate Jacobian radical must equal the intersection of the four stated primes, the power witnesses must cover every radical generator, and the recomputed residual ideals must equal the short formulas used in the final intersection. The kernel-difference verifier identifies the named generators with the complete normalized-support kernel. These checks ensure that the final calculation uses the certified intermediate objects.

All scheme-theoretic conclusions use exact calculations and geometric arguments; numerical fitting, floating-point linear algebra and agreement on sampled points are not used to establish them. The package records the Python and SymPy versions and the actual run outputs. Its manifest distinguishes rerunning a recorded exact certificate from rebuilding a finite model from all possible coefficients. These computations complement the arguments for descent, generic rank, flatness, properness, normality and the scope of the moduli functors. The package provides a reproducible record of a computer-assisted proof, rather than a proof-assistant formalization.